\documentclass[reqno,a4paper,11pt]{article}
\usepackage{amsmath,amsthm,dsfont,amsfonts,amssymb,fancyhdr}
\usepackage{enumerate}
\usepackage[usenames,dvipsnames]{color}
\usepackage{bbm,soul,supertabular,longtable,verbatim,extarrows}
\usepackage{titlesec}
\usepackage{graphicx}
\usepackage[titletoc,toc,title]{appendix}

\usepackage{caption}

\usepackage{float}
\usepackage{BOONDOX-cal}
\usepackage{cite}
\usepackage{tikz}
\usepackage{booktabs,multirow,makecell}
\usepackage{authblk}
\usepackage[titletoc]{appendix}
\usetikzlibrary[trees]

\usepackage{lipsum}
\usepackage{mathrsfs}
\usepackage{indentfirst}
\usepackage[colorlinks=true, allcolors=blue]{hyperref}
\usepackage[top=2.4cm,bottom=2.2cm,left=2.6cm,right=2cm]{geometry}

\newtheorem{theorem}{Theorem}[section]
\newtheorem{proposition}[theorem]{Proposition}
\newtheorem{lemma}[theorem]{Lemma}
\newtheorem{corollary}[theorem]{Corollary}

\newtheorem{conjecture}[theorem]{Conjecture}

\theoremstyle{definition}
\newtheorem{definition}[theorem]{Definition}
\newtheorem{remark}[theorem]{Remark}
\newtheorem{problem}[theorem]{Problem}

\soulregister\cite7
\soulregister\citep7
\soulregister\citet7
\soulregister\ref7
\soulregister\pageref7
\begin{document}
\title{\textbf{Recent progress on graphs with fixed smallest eigenvalue}}
\author[a,b]{Jack H. Koolen}
\author[c]{Meng-Yue Cao}
\author[a]{Qianqian Yang\footnote{Corresponding author.}}
\affil[a]{\footnotesize{School of Mathematical Sciences, University of Science and Technology of China, 96 Jinzhai Road, Hefei, 230026, Anhui, PR China.}}
\affil[b]{\footnotesize{CAS Wu Wen-Tsun Key Laboratory of Mathematics, University of Science and Technology of China, 96 Jinzhai Road, Hefei, Anhui, 230026, PR China}}
\affil[c]{\footnotesize{School of Mathematical Sciences, Beijing Normal University, 19 Xinjiekouwai Street, Beijing, 100875, PR China.}}

\maketitle
\pagestyle{plain}

\newcommand\blfootnote[1]{%
\begingroup
\renewcommand\thefootnote{}\footnote{#1}%
\addtocounter{footnote}{-1}%
\endgroup}
\blfootnote{2010 Mathematics Subject Classification. Primary 05C50. Secondary 05C22, 05C75, 05E30, 05D99, 11H06.}
\blfootnote{E-mail addresses: {\tt koolen@ustc.edu.cn} (J.H. Koolen), {\tt cmy1325@163.com} (M.-Y. Cao),  {\tt qqyang91@ustc.edu.cn} (Q. Yang).}

\begin{abstract}
We give a survey on graphs with fixed smallest eigenvalue, especially on graphs with large minimal valency and also on graphs with good structures. Our survey mainly consists of the following two parts:
\begin{enumerate}
  \item Hoffman graphs, the basic theory related to Hoffman graphs and the applications of Hoffman graphs to graphs with fixed smallest eigenvalue and large minimal valency;
  \item recent results on distance-regular graphs and co-edge regular graphs with fixed smallest eigenvalue and the characterizations of certain families of distance-regular graphs.
\end{enumerate}

At the end of the survey, we also discuss signed graphs with fixed smallest eigenvalue and present some new findings.
\end{abstract}

%%
%% LaTeX can automatically make a table of contents.  This is done by
%% uncommenting the following:
%%

%\tableofcontents

%%
%%  To enter text is easy. Just type it.  A blank line starts a new
%%  paragraph.
%%

%%%%%%%%%%%%%%%%%%%%%%%%%%%%%%%%%%%%%%%%%%%%%%%%%%%%%%%%%%%%%%%%%%%%%%

\section{Introduction}\label{Sec:intro}
All graphs mentioned in this paper are finite, undirected and simple. For undefined notations see \cite{BH2012}, \cite{code}
and \cite{GD01}. Unless we specify a different matrix, by an eigenvalue of a graph, we mean an eigenvalue of its adjacency matrix. Note that as the adjacency matrix of a graph is a symmetric real matrix, it is diagonalizable and all its eigenvalues are real.

In this paper, we mainly study the smallest eigenvalue of a graph. One of the oldest results is that for a connected graph, its smallest eigenvalue is in absolute value at most the largest eigenvalue and equality holds if and only if this graph is bipartite (see \cite[Proposition 3.1.1, Proposition 3.4.1]{BH2012}). This is a consequence of the Perron-Frobenius Theorem.

An important result by Hoffman gives a bound on the stability number $\alpha$ of a $k$-regular graph $G$ with order $n$ and smallest eigenvalue $\lambda_{\min}(G)$ as follows:
$$\alpha \leq \frac{n}{1+ \frac{k}{-\lambda_{\min}(G)}}$$
(unpublished; see \cite[Theorem 3.5.2]{BH2012}). We call this bound the \emph{Hoffman bound} or the \emph{ratio bound}. This bound has many applications, for example in extremal combinatorics. Godsil and Meagher \cite{Godsil.2016} used this bound to show many Erd\H{o}s-Ko-Rado theorems. This bound also gives a lower bound on the chromatic number $\chi$ of $G$, that is
$\chi \geq 1+ \frac{k}{-\lambda_{\min}(G)}$, as each color class of $G$ is a stable set.

Recently, Bramoull\'{e}, Kranton and D'Amours \cite{Bramoulle.2014} have shown that the equilibria of many economic systems only depend on the smallest eigenvalue of the underlying network.

The main topic of this paper is to survey the area of graphs with fixed smallest eigenvalue.  In Section \ref{sec:main results}, we describe the classical $1976$ result of Cameron, Goethals, Seidel and Shult \cite{Cameron}
characterizing graphs with smallest eigenvalue at least $-2$ and the corresponding $2018$ result of Koolen, Yang and Yang \cite{kyy3} for graphs with smallest eigenvalue at least $-3$.
Also in this section, we present two classical results of Hoffman  and related results by Woo and Neumaier \cite{Woo}, Yu \cite{yu} and Aharoni, Alon and Berger \cite{Aharoni.2016}. The first result of
Hoffman \cite{Hoff1977} shows the following: Let $G$ be a graph with smallest eigenvalue $\lambda_{\min}(G)$ and large minimal valency. If $\lambda_{\min}(G)>-2$, then $\lambda_{\min}(G)=-1$ and $G$ is a disjoint union of cliques; if $\lambda_{\min}(G)>-1-\sqrt{2}$, then $\lambda_{\min}(G)=-2$ and $G$ is a generalized line graph. The second result of Hoffman \cite{Hoff1973} shows that the smallest eigenvalue of a graph depends very much on its local structure.  For recent surveys on graphs with smallest
eigenvalue at least $-2$, we refer to Cvetkovi\'{c}, Rowlinson and Simi\'{c} \cite{DC} and \cite{Cvetkovic.2015}.
In Section \ref{sec:hoffman graphs}, we define Hoffman graphs and present the basic theory of Hoffman graphs. In Section \ref{sec:associated hoffman graphs}, we investigate graphs with bounded smallest eigenvalue and large minimal valency. In Section \ref{sec:drg} we discuss distance-regular graphs. We look at distance-regular graphs with fixed smallest eigenvalue and also at characterizations of certain families of distance-regular graphs. In Section \ref{sec:co-edge regular graphs} we discuss co-edge regular graphs. In Section \ref{sec:signed graph} we explore signed graphs and Seidel matrices. In Section \ref{sec:future work}, we give several problems on unsigned and signed graphs. In Appendix \ref{appendix}, we define $Q$-polynomial distance-regular graphs and their Terwilliger algebra.

Note that if a regular graph $G$ has smallest eigenvalue $\lambda_{\min}(G)$, then its complement $\overline{G}$ is also regular with second largest eigenvalue $-\lambda_{\min}(G)-1$. Henceforth we scratch the area of regular graphs with fixed second largest eigenvalue.

\subsection{Regular graphs with fixed second largest eigenvalue}\label{sec:alon-boppana}
There is a tremendous amount of literature about regular graphs with fixed second largest
eigenvalue. In this subsection, we give some highlights of this area, but we do not intend to survey the area. One of the reasons is that the area has quite a bit of different flavors from
the rest of this paper, and another is that there is so much literature that probably will make a small book.
We follow \cite[Chapter 4]{BH2012}. For more details, see also that chapter.

An {\em expander} is a (preferably sparse) graph with the property that the number of vertices at distance at most $1$ from any given (not too large) set $S$ of vertices is at least a fixed constant
$(>1)$ times the size of $S$. Expanders became famous, because of their role in sorting networks (cf. Ajtai, Koml\'{o}s and Szemer\'{e}di \cite{Ajtai.1983}). For a recent survey on expanders, see \cite{Hoory.2006}.

Let $G$ be a connected $k$-regular graph with distinct eigenvalues $k > \lambda_1 >\cdots>\lambda_t$. Let $\lambda:= \max\{ \lambda_1, -\lambda_t\}$.
It is shown in \cite[Proposition 4.3.1, Proposition 4.5.1]{BH2012}, also in \cite{Alon.1988}, that if the ratio $\frac{k}{\lambda}$ is larger, then the expansion properties of $G$ are better. Also if the ratio $\frac{k}{\lambda}$ is large, then the graph has good connectivity and randomness properties. For a survey on pseudo-random graphs, that is, graphs with good randomness properties, see \cite{Krivelevich.2006}.

A theorem by Alon and Boppana shows that the second largest eigenvalue of a $k$-regular graph can not be much smaller than $2\sqrt{k-1}$.
\begin{theorem}[{\cite[Alon-Boppana]{Alon.1986}}]\label{alon}
Let $k \geq 3$ be an integer. There exists a positive constant $C$ such that the second largest eigenvalue $\lambda_1$ of a $k$-regular graph of order $n$ satisfies
$$\lambda_1 \geq \sqrt{k-1}(1- C\frac{\ln{(k-1)}}{\ln{n}}).$$
\end{theorem}
Serre \cite{Serre.1997} has shown that for a $k$-regular graph, many of its eigenvalues are not much smaller than $2\sqrt{k-1}$.
\begin{theorem}
Fix $k \geq 1$. For each $\varepsilon >0$, there exists a positive constant $c = c(k, \varepsilon)$ such that for any $k$-regular graph $G$ of order $n$, the number of eigenvalues larger than $(2-\varepsilon)\sqrt{k-1}$ is at least $cn$.
\end{theorem}

There are many improvements and generalizations of the above two theorems with applications to coding theory and other areas, see for example \cite{Cioaba.2019}, \cite{Hholdt.2014} and \cite{Li.1996}.

Alon \cite{Alon.1986} conjectured that for fixed $k, \varepsilon >0$ and $n$ sufficiently large, a random $k$-regular graph of order $n$ has second largest eigenvalue at most
$2\sqrt{k-1}+\varepsilon$. This conjecture was shown to be true by Friedman \cite{Friedman.2008}.

These results show that the above mentioned ratio $\frac{k}{\lambda}$ can not be much larger than $\frac{k}{2\sqrt{k-1}}$. This leads us to define Ramanujan graphs. A {\em Ramanujan graph} is a connected $k$-regular graph such that any eigenvalue $\lambda \neq \pm k$ satisfies $|\lambda| \leq 2\sqrt{k-1}$.
Complete graphs are Ramanujan graphs. Note that a sparse non-bipartite Ramanujan graph is a particular good expander. It is of very great interest to construct infinite families of Ramanujan graphs with fixed valency $k$ and unbounded number of vertices.
There are several constructions known of infinite families of non-bipartite Ramanujan graphs,
for example, by Lubotzky, Phillips and Sarnak \cite{Lubotzky.1988}, Margulis \cite{Margulis.1988} and Morgenstern \cite{Morgenstern.1994}, but they are only known for particular $k$. Recently, Marcus, Spielman and Srivastava \cite{Marcus.2015} found for every $k \geq 3$ an infinite family of bipartite Ramanujan graphs with valency $k$ and unbounded number of vertices. It still remains an open problem to construct infinite families of non-bipartite $k$-regular Ramanujan graphs for all of $k\geq3$.

\section{An overview of the main results}\label{sec:main results}

In this section, we give an overview of the main results in the area of graphs with fixed smallest eigenvalue.
At the beginning of this section, we introduce some basic terminology.

Let $G $ be a graph with \emph{vertex set} $V(G)$ and \emph{edge set} $E(G)\subseteq\binom{V(G)}{2}$, where $V(G)$ is a finite set.
If $\{x, y\}$ is in $E(G)$, then we say that $x$ and $y$ are \emph{neighbors} or $x$ and $y$ are \emph{adjacent} and write $x \sim y$ in this case. We say that a vertex $x$ is \emph{incident} with an edge $e$ if $x \in e$.

The \emph{adjacency matrix} $A(G)$ of a graph $G$ is the square $(0,1)$-matrix with rows and columns are indexed by $V(G)$, such that the $(x,y)$-entry of $A(G)$ is $1$ if and only if $x$ and $y$ are adjacent. As $A(G)$ is a symmetric matrix, all its eigenvalues are real. The \emph{eigenvalues} of $G$ are just the eigenvalues of $A(G)$. We call the multiset of eigenvalues of $G$ with their multiplicities the \emph{spectrum} of $G$. We call two graphs \emph{cospectral} if they have the same spectrum.
In this paper, we are mainly interested in the smallest eigenvalue of $G$, which is denoted by $\lambda_{\min}(G)$.

\subsection{\texorpdfstring{$s$}{s}-Integrability of graphs}
Let $\Sigma$ be a finite set of vectors in $\mathbb{R}^n$. The {\em lattice} $\Lambda$ generated by $\Sigma$ is defined as
$$\Lambda :=\left\{ \sum_{\mathbf{v} \in \Sigma} \alpha_{\mathbf{v}}\mathbf{v} \mid \alpha_{\mathbf{v}} \in \mathbb{Z} \text{ for all }\mathbf{v} \in \Sigma\right\},$$
and denoted by $\langle\Sigma\rangle_{\mathbb{Z}}$.
The lattice $\Lambda$ is called {\em integral}, if the standard inner product $(\mathbf{v}_1, \mathbf{v}_2)$ is integral for all $\mathbf{v}_1, \mathbf{v}_2 \in \Sigma$. Let $Gr(\Sigma)$ be the matrix with rows and columns are indexed by the set $\Sigma$, such that the $(\mathbf{v}_1,\mathbf{v}_2)$-entry of $Gr(\Sigma)$ is $(\mathbf{v}_1, \mathbf{v}_2)$. Note that the lattice $\Lambda$ is integral if and only if the matrix $Gr(\Sigma)$ is an integral matrix.

Following Conway and Sloane \cite{CS}, we say that an integral lattice $\Lambda$ is {\em $s$-integrable} (for some positive integer $s$) if $\sqrt{s}\Lambda$ is a sublattice of a standard lattice, which is a lattice generated by a set of orthonormal vectors. This is equivalent with the condition that $sGr(\Sigma) = N^TN$ holds for some integral matrix $N$.

Let $G$ be a graph with $A(G)$ as its adjacency matrix. Then the matrix $B(G):= A(G)+\lceil-\lambda_{\min}(G)\rceil I$ is positive semidefinite and hence can be written as $B(G) = M^TM$ for some real matrix $M$. Denote by $\Lambda(G)$ the lattice generated by the columns of $M$. Note that the isomorphism class of $\Lambda(G)$ only depends on $B(G)$, not on $M$. For a positive integer $s$, we say that the graph $G$ is \emph{$s$-integrable} if the lattice $\Lambda(G)$ is $s$-integrable, or equivalently, $sB(G) = N^TN$ holds for some integral matrix $N$.
Note that if a graph $G$ is $s$-integrable and $t$-integrable, then $G$ is $(s+t)$-integrable. So a $1$-integrable graph is $s$-integrable for $s \geq 1$.

In 1976, Cameron et al. \cite{Cameron} showed the following result:
\begin{theorem}[cf.~{\cite[Theorem 4.3, Theorem 4.10]{Cameron}}]\label{thmcam}
If $G$ is a connected graph with $\lambda_{\min}(G) \geq -2$, then $G$ is $s$-integrable for any $s \geq 2$. Moreover, if $G$ has at least $37$ vertices, then $G$ is $1$-integrable.
\end{theorem}

A graph is a \emph{generalized line graph} if it is $1$-integrable with smallest eigenvalue at least $-2$.
Let $G$ be a graph. The \emph{line graph} of $G$, denoted by $L(G)$, is the graph with vertex set $E(G)$ such that edges $e$ and $f$ are adjacent in $L(G)$ if there is a unique vertex $x$ incident with $e$ and $f$ in $G$.
Let $N$ be the $|V(G)|\times|E(G)|$ $(0,1)$-matrix whose $(x,e)$-entry equals $1$ if and only if the vertex $x$ is incident with the edge $e$. Then $A(L(G)) + 2I = N^TN$. This means that $L(G)$ has smallest eigenvalue at least $-2$ and is $1$-integrable. But for a generalized line graph $H$, if an integral matrix $N'$ satisfies $A(H)+2I=(N')^TN'$, then $N'$ is a $(0,\pm1)$-matrix.

For more about graphs with smallest eigenvalue at least $-2$, we refer to \cite{DC} and \cite{Cvetkovic.2015}.

Later in 2018, Koolen et al. \cite{kyy3} studied the integrability of graphs with smallest eigenvalue at least $-3$ and proved that:
\begin{theorem}[cf.~{\cite[Theorem 1.3]{kyy3}}]\label{mthmkyy3}
There exists a positive constant $\kappa_1$ such that, if $G$ is a graph with $\lambda_{\min}(G)\geq-3$ and minimal valency at least $\kappa_1$, then $G$ is $s$-integrable for any $s\geq2$.
\end{theorem}

To prove the above theorem, they use Hoffman graphs as their main tool, which will be introduced in the next section.
\begin{remark}
\begin{enumerate}
\item After Theorem \ref{kyy3thm}, we give a sketch of the proof of Theorem \ref{mthmkyy3}.
\item It is known that $\kappa_1$ is at least $166$ by results of Koolen and Munemasa \cite{koolen.2020b}  and Koolen, Rehman and Yang \cite{KRY2}.
\end{enumerate}
\end{remark}

\subsection{Two results of Hoffman}

Hoffman \cite{Hoff1977} in 1977 showed the following related results.

\begin{theorem}[{cf.~\cite[Theorem 1.1]{Hoff1977}}]\label{thmhoff77}
\begin{enumerate}
\item For any real number $\lambda\in (-2,-1]$, there exists a constant $C_1(\lambda)$ such that, if $G$ is a graph with $\lambda_{\min}(G)\geq\lambda$ and minimal valency at least $C_1(\lambda)$, then $\lambda_{\min}(G) = -1$ and $G$ is a disjoint union of cliques.
\item For any real number $\lambda\in (-1-\sqrt{2},-2]$, there exists a constant $C_1(\lambda)$ such that, if $G$ is a graph with $\lambda_{\min}(G)\geq\lambda$ and minimal valency at least $C_1(\lambda)$, then $\lambda_{\min}(G) \geq -2$ and $G$ is a generalized line graph and hence is $1$-integrable.
\end{enumerate}
\end{theorem}

The second item of this theorem can be reformulated as follows:
\begin{theorem}[{cf.~\cite[Theorem 3.12.5]{bcn}}]
%\begin{enumerate}[(i)]
Let $\hat{\theta}_k$ be the supremum of the smallest eigenvalues of graphs with minimal valency at least $k$ and smallest eigenvalue less than $-2$. Then $\{\hat{\theta}_k\}_{k=1}^{\infty}$ forms a monotone decreasing sequence with limit $-1-\sqrt{2}$.
%\end{enumerate}
\end{theorem}

Following the ideas of Hoffman, Woo and Neumaier \cite{Woo} in 1995 showed that
\begin{theorem}[{cf.~\cite[Theorem 5.1]{Woo}}]
For any real number $\lambda\in (\alpha_1,-1-\sqrt{2}]$, where $\alpha_1\approx-2.4812$ is the smallest root of the polynomial $x^3+2x^2-2x-2$, there exists a positive constant $C_1(\lambda)$ such that, if $G$ is a graph with $\lambda_{\min}(G)\geq\lambda$ and minimal valency at least $C_1(\lambda)$, then $\lambda_{\min}(G)\geq-1-\sqrt{2}$.
\end{theorem}

This theorem can also be reformulated as follows:

\begin{theorem}[{\cite[Theorem 1.3]{yu}}]
Let $\hat{\sigma}_k$ be the supremum of the smallest eigenvalues of graphs with minimal valency at least $k$ and smallest eigenvalue less than $-1-\sqrt{2}$. Then $\{\hat{\sigma}_k\}_{k=1}^{\infty}$ forms a monotone decreasing sequence with limit $\alpha_1$, the smallest root of the polynomial $x^3+2x^2-2x-2$.
\end{theorem}

\begin{remark}
\begin{enumerate}[(i)]
\item Bussemaker and Neumaier \cite{BuNe} showed that $\hat{\theta}_1$ is the smallest eigenvalue of the graph $E_{10}$ (for a picture see below) and this graph is the unique connected graph with $\hat{\theta}_1$ as its smallest eigenvalue, which is approximately $-2.006594$, the smallest root of the polynomial $x^2(x^2-1)^2(x^2-3)(x^2-4)-1$.
 
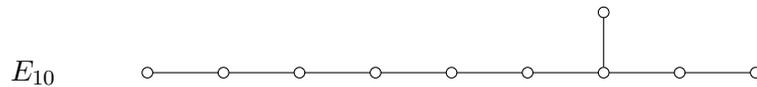
\begin{figure}[H]
    \centering
    \begin{tikzpicture}
    \draw (-1.5,0) node {$E_{10}$};
    \tikzstyle{every node}=[draw,circle,fill=white,minimum size=4pt,
                            inner sep=0pt]
                            {every label}=[\tiny]
   \draw (0,0) node (1d) [label=below:$$] {}
       -- ++(0:1cm) node (2d) [label=below:$$] {}
       -- ++(0:1cm) node (3d) [label=below:$$] {}
       -- ++(0:1cm) node (4d) [label=below:$$] {}
       -- ++(0:1cm) node (5d) [label=below:$$] {}
       -- ++(0:1cm) node (6d) [label=below:$$] {}
       -- ++(0:1cm) node (7d) [label=below:$$] {}
       -- ++(0:1cm) node (8d) [label=below:$$] {}
       -- ++(0:1cm) node (9d) [label=below:$$] {};
   \draw (7d)
       -- ++(90:0.8cm) node (10d) [label=right:$$] {};
    \end{tikzpicture}
\end{figure}

\item Let $\hat{\eta}_k$ be the supremum of the smallest eigenvalues of $k$-regular graphs with smallest eigenvalue less than $-2$. Then the limit of the sequence $\{\hat{\eta}_k\}_{k=1}^{\infty}$ is $-1-\sqrt{2}$ (see \cite{yu}).

\item Yu \cite{yu} also showed that $\hat{\eta}_3$ is the smallest eigenvalue of the Yu-graph (for a picture see below) and this graph is the unique connected $3$-regular graph with $\hat{\eta}_3$ as its smallest eigenvalue, which is approximately $-2.0391$, the smallest root of the polynomial
$x^6- 3 x^5 -7x^4 +21x^3 +13x^2 -35x -4$.
\begin{figure}[H]
    \centering
    \begin{tikzpicture}
    \draw (-1.5,0) node {Yu-graph};
    \tikzstyle{every node}=[draw,circle,fill=white,minimum size=4pt,
                            inner sep=0pt]
                            {every label}=[\tiny]
   \draw (0,0) node (1a) [label=below:$$] {}
       -- ++(60:1cm) node (2a) [label=below:$$] {}
       -- ++(0:1.2cm) node (3a) [label=below:$$] {}
       -- ++(-60:1cm) node (4a) [label=below:$$] {}
       -- ++(-120:1cm) node (5a) [label=below:$$] {}
       -- ++(180:1.2cm) node (6a) [label=below:$$] {}
       -- ++(1a);
   \draw (2a)
       -- ++(-60:1cm) node (7a) [label=right:$$] {}
       -- ++(6a);
   \draw (1a) -- (7a);
   \draw (3a) -- (5a);

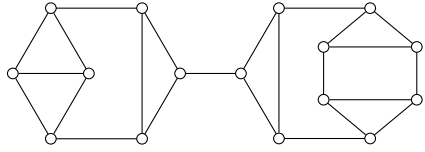
   \draw (3,0) node (1b) [label=below:$$] {}
       -- ++(60:1cm) node (2b) [label=below:$$] {}
       -- ++(0:1.2cm) node (3b) [label=below:$$] {}
       -- ++(-40:0.8cm) node (4b) [label=below:$$] {}
       -- ++(-90:0.7cm) node (5b) [label=below:$$] {}
       -- ++(-140:0.8cm) node (6b) [label=below:$$] {}
       -- ++(180:1.2cm) node (7b) [label=below:$$] {}
       -- ++(1b);
   \draw (3b)
       -- ++(-140:0.8cm) node (8b) [label=right:$$] {}
       -- ++(-90:0.7cm) node (9b) [label=below:$$] {}
       -- ++(6b);
   \draw (4a) -- (1b);
   \draw (2b) -- (7b);
   \draw (4b) -- (8b);
   \draw (5b) -- (9b);
    \end{tikzpicture}
\end{figure}
\item Let $\hat{\xi}_k$ be the supremum of the smallest eigenvalues of $k$-regular graphs with smallest eigenvalue less than $-1-\sqrt{2}$. Then the limit of the sequence $\{\xi_k\}_{k=1}^{\infty}$ is $\alpha_1$ (see \cite{yu}).
\end{enumerate}
\end{remark}

Now we look at graphs with a bounded smallest eigenvalue. In order to state the results we need to introduce the graph $\widetilde{K}_{2t}$. Let $t$ be a positive integer. Denote by $\widetilde{K}_{2t}$ the graph with $2t+1$ vertices consisting of a clique $K_{2t}$ together with a vertex that is adjacent to exactly $t$ vertices of the clique. Note that the smallest eigenvalue of $\widetilde{K}_{2t}$ goes to $-\infty$ as $t$ goes to $\infty$. (For a proof, see \cite[Lemma 3.2]{Yang.2021}.) It is fairly easy to see that the smallest eigenvalue of the $t$-claw $K_{1,t}$ also goes to $-\infty$ as $t$ goes to $\infty$, as $\lambda_{\min}(K_{1,t})=-\sqrt{t}$.

Hoffman \cite{Hoff1973} in 1973 showed the following results:

\begin{theorem}[{\cite[p.~278]{Hoff1973}}]\label{Hoff1973}
\begin{enumerate}
\item Let $\lambda$ be a positive real number. Then there exists a positive integer $T = T(\lambda)$ such that, if $G$ is a graph with $\lambda_{\min}(G) \geq -\lambda$, then $G$ contains neither $K_{1,T}$ nor $\widetilde{K}_{2T}$ as an induced subgraph.
\item Let $t\geq 3$ be a positive integer. Then there exists a positive constant $\lambda = \lambda(t)$ such that, if a graph $G$ does not contain $K_{1,t}$ and $\widetilde{K}_{2t}$, then $\lambda_{\min}(G) \geq -\lambda$.
\end{enumerate}
\end{theorem}

This theorem shows that the smallest eigenvalue of a graph is quite dependent on the local information.

A related result is obtained by Aharoni et al. \cite{Aharoni.2016}. Before we state their result we need to introduce the Laplacian matrix. For a graph $G$ on $n$ vertices, define its \emph{Laplacian matrix} $L(G)$ as follows: $L(G)= \Delta(G) - A(G)$, where $\Delta(G)$ is the diagonal matrix with $\Delta(G)_{x,x} = k_G(x)$, the valency of $x$ in $G$, and $A(G)$ is the adjacency matrix of $G$.
They obtained the following result:
\begin{theorem}[{\cite[Theorem 1.1]{Aharoni.2016}}]\label{aharoni}
Let $G$ be a graph with maximal valency $k$ containing no induced $K_{1,\ell}$'s. Let $t(k, \ell)$ denote the minimum possible number of edges of a graph on $k$ vertices with no stable set of size $\ell$. If $\theta$ is the maximal eigenvalue of the Laplacian matrix of $G$, then $\theta \leq 2k - \frac{t(k, \ell)}{k-1}$.
\end{theorem}
It is known by Tur\'{a}n's Theorem that $t(k , \ell) = (1+o(1))\frac{k^2}{2\ell-2}$, where the $o(1)$-term tends to zero as $k$ tends to infinity.

For regular graphs, we obtain the following corollary:
\begin{corollary}\label{bound on se of regular graphs}
Let $G$ be a $k$-regular graph containing no induced $K_{1, \ell}$'s. Then $\lambda_{\min}(G)\geq -(1+o(1))\frac{2\ell-3}{2\ell-2}k$.
\end{corollary}
\begin{proof}
This follows from Theorem \ref{aharoni} immediately.
\end{proof}
It is not clear whether the lower bound in Corollary \ref{bound on se of regular graphs} is the best bound on the smallest eigenvalue of regular graphs. Here, given an integer $\ell\geq 2$ and a connected bipartite $(\ell-1)$-regular graph $G'$, we are able to construct an infinite family of regular graphs $\{G_1,G_2,\ldots,G_s,\ldots\}$ satisfying
\begin{enumerate}
\item $G_s$ is $k_s$-regular and does not contain induced $K_{1,\ell}$'s, for $s=1,2,\ldots$;
\item $k_s\rightarrow\infty$, as $s\rightarrow\infty$;
\item $\lim_{s\rightarrow \infty}\frac{\lambda_{\min}(G_s)}{k_s}=-\frac{\ell -2}{\ell}$.
\end{enumerate}

For this, we need to introduce the clique extension of a given graph. For a positive integer $s$ and a graph $H$, the $s$-\emph{clique extension} of $H$ is the graph $\widetilde{H}$ obtained from $H$ by replacing each vertex $x\in V(H)$ by a clique $\widetilde{X}$ with $s$ vertices, such that $\tilde{x}\sim\tilde{y}$ (for $\tilde{x}\in\widetilde{X},~\tilde{y}\in\widetilde{Y},~\widetilde{X}\neq\widetilde{Y}$) in $\widetilde{H}$ if and only if $x\sim y$ in $H$. In particular, if $H$ has spectrum
\begin{equation}\label{spectrum1}
 \left\{\lambda_0^{m_0},\lambda_1^{m_1},\ldots,\lambda_t^{m_t}\right\},
 \end{equation}
then $\widetilde{H}$ has spectrum
 \begin{equation}\label{spectrum2}
\left\{(s(\lambda_0+1)-1)^{m_0},(s(\lambda_1+1)-1)^{m_1},\ldots,(s(\lambda_t+1)-1)^{m_t},(-1)^{(m_0+m_1+\cdots+m_t)}\right\}
 \end{equation}
(see \cite[p.~107]{Hayat.2019}).

Now we start our construction. For each $s$, let $G_s$ be the $s$-clique extension of $G'$. Then $G_s$ is regular with valency $k_s:=\ell s-1$ and with smallest eigenvalue $\lambda_{\min}(G_s):=-\ell s +2s-1$. It is not hard to check that these graphs $G_s$'s satisfy the above properties.

Let $\lambda_{\ell,k}:=\inf\left\{\lambda_{\min}(G)\mid G\right.$ is $k$-regular and does not contain induced $K_{1,\ell}$'s$\}$ be a real number. Consider \[\tau_{\ell}:=\inf\left\{\frac{\lambda_{\ell,k}}{k}\mid k=3,4,\ldots\right\}.\]
Corollary \ref{aharoni} and the above examples we constructed imply $-\frac{2\ell-3}{2\ell-2}\leq \tau_{\ell}\leq -\frac{\ell -2}{\ell}$.
\begin{problem}
Determine $\tau_{\ell}$ for all $\ell\geq3$.
\end{problem}
For the particular case where $\ell=3$, Cioabă, Elzinga and Gregory \cite[Theorem 4.5]{Cioaba.2020} proved $\lambda_{3,3}\geq\theta\approx-2.272$, where $\theta$ is the smallest root of the polynomial $x^3+x+14$. Furthermore, for graphs containing no induced $K_{1,3}$'s, that is, claw-free graphs, Chudnovsky and Seymour \cite{Chudnovsky.2005,Chudnovsky.2007,Chudnovsky.2008,Chudnovsky.2008b,Chudnovsky.2008c,Chudnovsky.2008d,Chudnovsky.2010,Chudnovsky.2012},
developed a structure theory. This may help to determine $\tau_2$.

\section{Hoffman graphs}\label{sec:hoffman graphs}

We describe now two methods used to study graphs with fixed smallest eigenvalues. The first technique is the so-called Bose-Laskar method. In this method, they use the fact that if a graph does not contain induced $t$-claws with large $t$, then this graph must contain large cliques. This method works best if there is some local regularity in the graph. The second method is to use Hoffman graphs as a tool, which was introduced by Woo and Neumaier \cite{Woo}, following ideas of Hoffman. This method gives more precise local information of a graph than the Bose-Laskar method, but the disadvantage of this method is that we need to assume that the minimal valency of graphs is very large, as we need Ramsey theory to show the existence of large cliques. In this section we discuss Hoffman graphs and give the basic theory for them.

\begin{definition}[Hoffman graph] A Hoffman graph $\mathfrak{h}$ is a pair $(H,\ell)$, where $H=(V,E)$ is a graph and $\ell:V \to\{f,s\}$ is a labeling map satisfying the following conditions:
\begin{enumerate}
\item vertices with label $f$ are pairwise non-adjacent,
\item every vertex with label $f$ is adjacent to at least one vertex with label $s$.
\end{enumerate}
\end{definition}
We call a vertex with label $s$ a \emph{slim vertex}, and a vertex with label $f$ a \emph{fat vertex}. We denote by $V_{\mathrm{slim}}(\mathfrak{h})$ (resp. $V_{\mathrm{fat}}(\mathfrak{h})$) the set of slim (resp. fat) vertices of $\mathfrak{h}$.

\vspace{0.1cm}
For a vertex $x$ of $\mathfrak{h}$, we define $N_{\mathfrak{h}}^{s}(x)$ (resp. $N_{\mathfrak{h}}^{f}(x)$) the set of slim (resp. fat) neighbors of $x$ in $\mathfrak{h}$. If every slim vertex of $\mathfrak{h}$ has a fat neighbor, then we call $\mathfrak{h}$ \emph{fat}, and if every slim vertex of $\mathfrak{h}$ has at least $t$ fat neighbors, we call $\mathfrak{h}$ $t$-\emph{fat}. In a similar fashion, we define $N^{f}_\mathfrak{h}(x_1,x_2)$ to be the set of common fat neighbors of two slim vertices $x_1$ and $x_2$ in $\mathfrak{h}$ and $N^{s}_\mathfrak{h}(f_1,f_2)$ to be the set of common slim neighbors of two fat vertices $f_1$ and $f_2$ in $\mathfrak{h}$.

\vspace{0.1cm}
The \emph{slim graph} of the Hoffman graph $\mathfrak{h}$ is the subgraph of $H$ induced on $V_{\mathrm{slim}}(\mathfrak{h})$. Note that any graph can be considered as a Hoffman graph with only slim vertices, and vice versa. We will not distinguish between Hoffman graphs with only slim vertices and graphs.

\vspace{0.1cm}
A Hoffman graph $\mathfrak{h}_1= (H_1, \ell_1)$ is called an (\emph{proper}) \emph{induced Hoffman subgraph} of $\mathfrak{h}=(H, \ell)$, if $H_1$ is an (proper) induced subgraph of $H$ and $\ell_1(x) = \ell(x)$ holds for all vertices $x$ of $H_1$.

\vspace{0.1cm}
Let $W$ be a subset of $V_{\mathrm{slim}}(\mathfrak{h})$. An induced Hoffman subgraph of $\mathfrak{h}$ generated by $W$, denoted by $\langle W\rangle_{\mathfrak{h}}$, is the Hoffman subgraph of $\mathfrak{h}$ induced on $W \cup\{f\in V_{\mathrm{fat}}(\mathfrak{h})\mid f \sim w \text{ for some }w\in W \}$.

\vspace{0.1cm}
For a fat vertex $f$ of $\mathfrak{h}$, a \emph{quasi-clique} (with respect to $f$) is a subgraph of the slim graph of $\mathfrak{h}$ induced on the slim vertices adjacent to $f$ in $\mathfrak{h}$, and we denote it by $Q_{\mathfrak{h}}(f)$.
\begin{definition}[isomorphism of Hoffman graphs]Two Hoffman graphs $\mathfrak{h}= (H, \ell)$ and $\mathfrak{h}^\prime=(H^\prime, \ell^\prime)$ are isomorphic if there exists an isomorphism from $H$ to $H^\prime$ which preserves the labeling.
\end{definition}

\begin{definition}[strong isomorphism of Hoffman graphs]Two Hoffman graphs $\mathfrak{h}= (H, \ell)$ and $\mathfrak{h}^\prime=(H^\prime, \ell^\prime)$ are strongly isomorphic if they have the same set of slim vertices and there exists an isomorphism from $\mathfrak{h}$ to $\mathfrak{h}^\prime$ which fixes the set of slim vertices vertex-wise.
\end{definition}

Note that if two Hoffman graphs are strongly isomorphic, then they have the same slim graph.

\begin{definition}[special matrix] For a Hoffman graph $\mathfrak{h}=(H,\ell)$, there exists a matrix $C$ such that the adjacency matrix $A$ of $H$ satisfies
\begin{eqnarray*}
A=\left(
\begin{array}{cc}
A_s  & C\\
C^{T}  & O
\end{array}
\right),
\end{eqnarray*}
where $A_s$ is the adjacency matrix of the slim graph of $\mathfrak{h}$. The special matrix $Sp(\mathfrak{h})$ of $\mathfrak{h}$ is the real symmetric matrix $A_s-CC^{T}.$
\end{definition}

The \emph{eigenvalues} of $\mathfrak{h}$ are the eigenvalues of its special matrix $Sp(\mathfrak{h})$, and the smallest eigenvalue of $\mathfrak{h}$ is denoted by  $\lambda_{\min}(\mathfrak{h})$. Note that $\mathfrak{h}$ is not determined by its special matrix in general, since different $\mathfrak{h}$'s may have the same special matrix. Observe also that if there are no fat vertices in $\mathfrak{h}$, then $Sp(\mathfrak{h})=A_s$ is just the standard adjacency matrix.

\begin{lemma}[{\cite[Lemma 3.4]{Woo}}]\label{fatnbr}
Let $\mathfrak{h}$ be a Hoffman graph and let $x_i$ and $x_j$ be two distinct slim vertices of $\mathfrak{h}$. The special matrix $Sp(\mathfrak{h})$ has diagonal entries

$$Sp(\mathfrak{h})_{x_i,x_i} = - |N_\mathfrak{h}^f (x_i)|$$
and off-diagonal entries
$$Sp(\mathfrak{h})_{x_i,x_j} = (A_s)_{x_i,x_j} -|N_\mathfrak{h}^f (x_i,x_j)|.$$
\end{lemma}

For the smallest eigenvalues of Hoffman graphs and their induced Hoffman subgraphs, Woo and Neumaier showed the following inequality.
\begin{lemma}[{\cite[Corollary 3.3]{Woo}}]\label{hoff}
If $\mathfrak{h}_1$ is an induced Hoffman subgraph of a Hoffman graph $\mathfrak{h}$, then $\lambda_{\min}(\mathfrak{h}_1)\geq\lambda_{\min}(\mathfrak{h})$ holds.
\end{lemma}

As a corollary of Lemma \ref{hoff}, we have:
\begin{lemma}\label{interelacing}
If $G_1$ is an induced subgraph of $G$, then $\lambda_{\min}(G_1)\geq\lambda_{\min}(G)$ holds.
\end{lemma}

\begin{definition}[$\mu$-saturated Hoffman graph]
Let $\mu\leq-1$ be a real number and let $\mathfrak{h}$ be a Hoffman graph with smallest eigenvalue at least $\mu$. Then $\mathfrak{h}$ is $\mu$-saturated if no fat vertex can be attached to $\mathfrak{h}$ in such a way that the resulting Hoffman graph has smallest eigenvalue at least $\mu$.
\end{definition}

Now we introduce a result of Hoffman and Ostrowski. In order to state this, we need to introduce the following notations.  Suppose $\mathfrak{h}$ is a Hoffman graph and $\{f_1, \dots, f_r\}$ is a subset of $V_{\mathrm{fat}}(\mathfrak{h})$. Let $\mathfrak{g}^{n_1, \ldots, n_r}(\mathfrak{h})$ be the Hoffman graph obtained from $\mathfrak{h}$ by replacing the fat vertex $f_i$ by a slim $n_i$-clique $K^{f_i}$, and joining all the neighbors of $f_i$ (in $\mathfrak{h}$) with all the vertices of $K^{f_i}$  for all $i$. We will write $G(\mathfrak{h},n)$ for the graph $\mathfrak{g}^{n_1, \ldots, n_r}(\mathfrak{h})$, when $V_{\mathrm{fat}}(\mathfrak{h}) = \{ f_1, f_2, \ldots, f_r\}$ and $n_1 =n_2 = \dots = n_r = n$. With the above notations, we can now state the result of Hoffman and Ostrowski. For a proof of it, see \cite[Theorem 2.14]{HJAT}.

\begin{theorem}\label{Ostrowski} Suppose $\mathfrak{h}$ is a Hoffman graph with fat vertices $f_1, f_2, \dots, f_r $.Then
\begin{center}$\lambda_{\min}(\mathfrak{g}^{n_1,\dots, n_r}(\mathfrak{h}))\geq \lambda_{\min}(\mathfrak{h})$,\end{center}
and
\begin{displaymath}\lim_{n_1,\dots,n_r \rightarrow \infty}\lambda_{\min}(\mathfrak{g}^{n_1,\dots, n_r}(\mathfrak{h}))= \lambda_{\min}(\mathfrak{h}).\end{displaymath}
\end{theorem}

\begin{definition}[representation of Hoffman graphs]
For a Hoffman graph $\mathfrak{h}$ and a positive integer $m$, a mapping $\phi: V(\mathfrak{h}) \to \mathbb{R}^m$ (resp.~$\phi: V(\mathfrak{h}) \to \mathbb{Z}^m$) satisfying
\[
(\phi(x),\phi(y))=\left\{
\begin{array}{ll}
  t & \text{if } x=y \text{ and } x,y\in V_{\mathrm{slim}}(\mathfrak{h}); \\
  1 & \text{if } x=y \text{ and } x,y\in V_{\mathrm{fat}}(\mathfrak{h}); \\
  1 & \text{if } x\sim y; \\
  0 & \text{otherwise},
\end{array}
\right.
\]
is a (resp.~integral) representation of $\mathfrak{h}$ of norm $t$.
\end{definition}
We denote by $\Lambda(\mathfrak{h},t)$ the lattice generated by the set $\{\phi(x)\mid x \in V(\mathfrak{h})\}$. Note that the isomorphism class of $\Lambda(\mathfrak{h},t)$ depends only on $\mathfrak{h}$ and $t$, and is independent of $\phi$, justifying the notation.

\begin{definition}[reduced representation of Hoffman graphs]\label{reducedrep}
For a Hoffman graph $\mathfrak{h}$ and a positive integer $m$, a mapping $\psi: V_{\mathrm{slim}}(\mathfrak{h}) \to \mathbb{R}^m$ (resp.~$\phi: V(\mathfrak{h}) \to \mathbb{Z}^m$) satisfying
\[
(\psi(x),\psi(y))=\left\{
\begin{array}{ll}
  t-|N_\mathfrak{h}^f(x)| & \text{if } x=y; \\
  1-|N_\mathfrak{h}^f(x,y)| & \text{if } x\sim y; \\
  -|N_\mathfrak{h}^f(x,y)| & \text{otherwise},
\end{array}
\right.
\]
is a (resp.~integral) reduced representation of $\mathfrak{h}$ of norm $t$.
\end{definition}
We denote by $\Lambda^{\text{red}}(\mathfrak{h},t)$ the lattice generated by the set $\{\psi(x)\mid x \in V_{\mathrm{slim}}(\mathfrak{h})\}$. Note that the isomorphism class of $\Lambda^{\text{red}}(\mathfrak{h},t)$ also depends only on $\mathfrak{h}$ and $t$, and is independent of $\psi$, justifying the notation.

\begin{lemma}[{\cite[Theorem $2.8$]{HJAT}}]\label{rela}
For a Hoffman graph $\mathfrak{h}$, the following conditions are equivalent:
\begin{enumerate}
\item $\mathfrak{h}$ has a representation of norm $t$;
\item $\mathfrak{h}$ has a reduced representation of norm $t$;
\item $\lambda_{\min}(\mathfrak{h})\geq -t$.
\end{enumerate}
\end{lemma}

A Hoffman graph $\mathfrak{h}$ is called \emph{integrally representable of norm $t$}, if $\mathfrak{h}$ has an integral representation $\phi: V(\mathfrak{h}) \rightarrow \mathbb{Z}^{m}$ of norm $t$ for some $m$.

\subsection{Sum and decomposition}
\begin{definition}[sum]\label{directsummatrix}
Let $\mathfrak{h}^1$ and $\mathfrak{h}^2$ be two Hoffman graphs. A Hoffman graph $\mathfrak{h}$ is the sum of $\mathfrak{h}^1$ and $\mathfrak{h}^2$, denoted by $\mathfrak{h} =\mathfrak{h}^1\uplus\mathfrak{h}^2$, if $\mathfrak{h}$ satisfies the following condition:

There exists a partition $\big\{V_{\mathrm{slim}}^1(\mathfrak{h}),V_{\mathrm{slim}}^2(\mathfrak{h})\big\}$ of $V_{\mathrm{slim}}(\mathfrak{h})$ such that induced Hoffman subgraphs generated by $V_{\mathrm{slim}}^i(\mathfrak{h})$ are $\mathfrak{h}^i$ for $i=1,2$ and
\[Sp(\mathfrak{h})=
\begin{pmatrix}
Sp(\mathfrak{h}^1) & O \\
O& Sp(\mathfrak{h}^2)
\end{pmatrix}
\] with respect to the partition $\big\{V_{\mathrm{slim}}^1(\mathfrak{h}),V_{\mathrm{slim}}^2(\mathfrak{h})\big\}$ of $V_{\mathrm{slim}}(\mathfrak{h})$.
\end{definition}

Clearly, by definition, the sum is associative, so that the sum $\biguplus_{i=1}^{r}\mathfrak{h}^i$ is well-defined. We can check that $\mathfrak{h}$ is a sum of two non-empty Hoffman graphs if and only if $Sp(\mathfrak{h})$ is a block matrix with at least two blocks. If $\mathfrak{h} =\mathfrak{h}^1 \uplus \mathfrak{h}^2$ for some non-empty Hoffman subgraphs $\mathfrak{h}^1$ and $\mathfrak{h}^2$, then we call $\mathfrak{h}$ \emph{decomposable} with $\{\mathfrak{h}^1,\mathfrak{h}^2\}$ as a \emph{decomposition} and call $\mathfrak{h}^1, \mathfrak{h}^2$ \emph{factors} of $\mathfrak{h}$. Otherwise, $\mathfrak{h}$ is called \emph{indecomposable}.

The following lemma gives a combinatorial way to define the sum of Hoffman graphs.

\begin{lemma}[{\cite[Lemma 2.11]{kyy1}}]\label{combi}
Let $\mathfrak{h}$ be a Hoffman graph and $\mathfrak{h}^1$ and $\mathfrak{h}^2$ be two induced Hoffman subgraphs of $\mathfrak{h}$. The Hoffman graph $\mathfrak{h}$ is the sum of $\mathfrak{h}^1$ and $\mathfrak{h}^2$ if and only if $\mathfrak{h}^1$, $\mathfrak{h}^2$, and $\mathfrak{h}$ satisfy the following conditions:
\begin{enumerate}
\item $V(\mathfrak{h})=V(\mathfrak{h}^1)\cup V(\mathfrak{h}^2);$
\item $\big\{V_{\mathrm{slim}}(\mathfrak{h}^1),V_{\mathrm{slim}}(\mathfrak{h}^2)\big\}$ is a partition of $V_{\mathrm{slim}}(\mathfrak{h});$
\item if $x \in V_{\mathrm{slim}}(\mathfrak{h}^i),~f \in V_{\mathrm{fat}}(\mathfrak{h})$ and $x\sim f$, then $f\in V_{\mathrm{fat}}(\mathfrak{h}^i);$
\item if $x \in V_{\mathrm{slim}}(\mathfrak{h}^1)$ and $y \in V_{\mathrm{slim}}(\mathfrak{h}^2)$, then $x$ and $y$ have at most one common fat neighbor, and they have one if and only if they are adjacent.
\end{enumerate}
  \end{lemma}

Let $\mu\leq -1$ be a real number and $\mathfrak{h}$ a Hoffman graph with $\lambda_{\min}(\mathfrak{h})\geq \mu$. The Hoffman graph $\mathfrak{h}$ is said to be \emph{$\mu$-reducible} if there exists a Hoffman graph $\widetilde{\mathfrak{h}}$ containing $\mathfrak{h}$ as an induced Hoffman subgraph, such that there is a decomposition $\{\widetilde{\mathfrak{h}}_{i}\}_{i=1}^{2}$ of $\widetilde{\mathfrak{h}}$ with $\lambda_{\min}(\widetilde{\mathfrak{h}}_{i})\geq \mu$ and $V_{s}(\widetilde{\mathfrak{h}}_{i})\cap V_{s}(\mathfrak{h})\neq \emptyset\ (i=1,2)$. We say that $\mathfrak{h}$ is \emph{$\mu$-irreducible} if $\lambda_{\min}(\mathfrak{h})\geq \mu$ and $\mathfrak{h}$ is not $\mu$-reducible. A Hoffman graph $\mathfrak{h}$ is said to be \emph{reducible} if $\mathfrak{h}$ is $\lambda_{\min}(\mathfrak{h})$-reducible. We say $\mathfrak{h}$ is \emph{irreducible} if $\mathfrak{h}$ is not reducible.

\begin{definition}[line Hoffman graph]
Let $\mathfrak{H}$ be a family of pairwise non-isomorphic Hoffman graphs. A Hoffman graph $\mathfrak{h}$ is an $\mathfrak{H}$-line Hoffman graph if there exists a Hoffman graph $\mathfrak{h}^\prime$ satisfying the following conditions:
\begin{enumerate}
\item $\mathfrak{h}^\prime$ has $\mathfrak{h}$ as an induced Hoffman subgraph;
\item $\mathfrak{h}^\prime$ has the same slim graph as $\mathfrak{h}$;
\item $\mathfrak{h}^\prime=\biguplus_{i=1}^{r}\mathfrak{h}_i^\prime$, where $\mathfrak{h}_i^\prime$ is isomorphic to an induced Hoffman subgraph of some Hoffman graph in $\mathfrak{H}$ for $i=1,\dots,r$.
\end{enumerate}
\end{definition}

\begin{definition}[$\mathfrak{H}$-saturated Hoffman graph]
Let $\mathfrak{H}$ be a family of pairwise non-isomorphic Hoffman graphs. A Hoffman graph $\mathfrak{h}$ is $\mathfrak{H}$-saturated, if $\mathfrak{h}$ is an $\mathfrak{H}$-line Hoffman graph, and no fat vertex can be attached to $\mathfrak{h}$ in such a way that the resulting Hoffman graph is also an $\mathfrak{H}$-line Hoffman graph.
\end{definition}
Note that if we set $\mathfrak{H}$ to be the family of pairwise non-isomorphic $\mu$-irreducible Hoffman graphs, then a $\mu$-saturated Hoffman graph is $\mathfrak{H}$-saturated.

\subsection{\texorpdfstring{$\mathfrak{H}$}{\mathfrak{H}}-Saturated Hoffman graphs}

Now we depict several Hoffman graphs as follows. They appeared in \cite{Woo} for the first time. Actually some of them are not used in this paper, but we use the same symbols as in \cite{Woo} to avoid confusion.

\begin{figure}[H]
   \centering
    \begin{tikzpicture}
    \draw (-7,5) node {$\mathfrak{h}_1$ =};
    \draw (-2,5) node {$\mathfrak{h}_2$ =};
    \draw (3,5) node {$\mathfrak{h}_3$ =};
    \tikzstyle{every node}=[draw,circle,fill=black,minimum size=20pt,
                            inner sep=0pt]
                            {every label}=[\tiny]
    \draw (-6,4.5) node (1f1) [label=below:$$] {};
    \draw (-1,4.5) node (2f1) [label=below:$$] {};
    \draw (0.5,4.5) node (2f2) [label=below:$$] {};
    \draw (5,4.5) node (3f1) [label=below:$$] {};

    \tikzstyle{every node}=[draw,circle,fill=black,minimum size=8pt,
                            inner sep=0pt]
                            {every label}=[\tiny]

    \draw (-6,5.5) node (1s1) [label=below:$$] {};
    \draw (-0.25,5.5) node (2s1) [label=below:$$] {};
    \draw (4,5.5) node (3s1) [label=below:$$] {};
    \draw (6,5.5) node (3s2) [label=below:$$] {};

    \draw (1s1) -- (1f1);
    \draw (2f1) -- (2s1) -- (2f2);
    \draw (3s1) -- (3f1) -- (3s2);
    \end{tikzpicture}
   \vspace{-0.8cm}
\end{figure}
\begin{figure}[H]
   \centering
    \begin{tikzpicture}
    \draw (-7,2.5) node {$\mathfrak{h}_4$ =};
    \draw (-2,2.5) node {$\mathfrak{h}_5$ =};
    \draw (3,2.5) node {$\mathfrak{h}_6$ =};
    \tikzstyle{every node}=[draw,circle,fill=black,minimum size=20pt,
                            inner sep=0pt]
                            {every label}=[\tiny]

    \draw (-6,2) node (4f1) [label=below:$$] {};
    \draw (-4.5,2) node (4f2) [label=below:$$] {};
    \draw (0,2) node (5f1) [label=below:$$] {};
    \draw (5.5,2) node (6f1) [label=below:$$] {};
    \draw (7,2) node (6f2) [label=below:$$] {};

    \tikzstyle{every node}=[draw,circle,fill=black,minimum size=8pt,
                            inner sep=0pt]
                            {every label}=[\tiny]

    \draw (-6,3) node (4s1) [label=below:$$] {};
    \draw (-4.5,3) node (4s2) [label=below:$$] {};
    \draw (-1,3) node (5s1) [label=below:$$] {};
    \draw (0,3) node (5s2) [label=below:$$] {};
    \draw (1,3) node (5s3) [label=below:$$] {};
    \draw (4,3) node (6s1) [label=below:$$] {};
    \draw (5.5,3) node (6s2) [label=below:$$] {};
    \draw (7,3) node (6s3) [label=below:$$] {};

    \draw (4f1) -- (4s1) -- (4s2) -- (4f2);
    \draw (5s1) -- (5f1) -- (5s2) -- (5s3) -- (5f1);
    \draw (6s1) -- (6f1) -- (6s2) -- (6s3) -- (6f2);
    \end{tikzpicture}
       \vspace{-0.8cm}
\end{figure}
\begin{figure}[H]
   \centering
    \begin{tikzpicture}
    \draw (-7,0) node {$\mathfrak{h}_7$ =};
    \draw (-2,0) node {$\mathfrak{h}_8$ =};
    \draw (3,0) node {$\mathfrak{h}_9$ =};
    \tikzstyle{every node}=[draw,circle,fill=black,minimum size=20pt,
                            inner sep=0pt]
                            {every label}=[\tiny]

    \draw (-6,0) node (7f1) [label=below:$$] {};
    \draw (-3.586,0) node (7f2) [label=below:$$] {};
    \draw (-1,0) node (8f1) [label=below:$$] {};
    \draw (1.707,0.707) node (8f2) [label=below:$$] {};
    \draw (1.707,-0.707) node (8f3) [label=below:$$] {};
    \draw (4,0) node (9f1) [label=below:$$] {};
    \draw (6.707,0.707) node (9f2) [label=below:$$] {};
    \draw (6.707,-0.707) node (9f3) [label=below:$$] {};

    \tikzstyle{every node}=[draw,circle,fill=black,minimum size=8pt,
                            inner sep=0pt]
                            {every label}=[\tiny]

    \draw (-5.293,0.707) node (7s1) [label=below:$$] {};
    \draw (-5.293,-0.707) node (7s2) [label=below:$$] {};
    \draw (-4.586,0) node (7s3) [label=below:$$] {};
    \draw (0,0) node (8s1) [label=below:$$] {};
    \draw (0.707,0.707) node (8s2) [label=below:$$] {};
    \draw (0.707,-0.707) node (8s3) [label=below:$$] {};
    \draw (4.707,0.707) node (9s1) [label=below:$$] {};
    \draw (4.707,-0.707) node (9s2) [label=below:$$] {};
    \draw (5.707,0.707) node (9s3) [label=below:$$] {};
    \draw (5.707,-0.707) node (9s4) [label=below:$$] {};

    \draw (7f1) -- (7s1) -- (7s3) -- (7s2) -- (7f1);
    \draw (7s1) -- (7s2);
    \draw (7s3) -- (7f2);
    \draw (8f1) -- (8s1);
    \draw (8f2) -- (8s2) -- (8s1) -- (8s3) -- (8f3);
    \draw (9f2) -- (9s3) -- (9s1) -- (9f1) -- (9s2) -- (9s4) -- (9f3);
    \draw (9s3) -- (9s4);
   \end{tikzpicture}
   {\caption{}\label{fig:-1-sqrt2}}
\end{figure}

We first state a classical result by Krausz \cite{Krausz.1943}.
\begin{theorem} A graph $G$ of order $n$ is a line graph if and only if one can partition the edge-set of $G$ into cliques $\{C_1, C_2, \ldots, C_t\}$ such that each vertex lies in at
 most $2$ $C_i$'s. Moreover, if $G$ is a connected line graph and {\bf  $n \geq 7$}, then this partition into cliques is unique.
\end{theorem}
In terms of line Hoffman graphs, we can formulate this result as follows.

\begin{theorem} Every $\{\mathfrak{h}_2\}$-line Hoffman graph whose slim graph is connected of order at least {\bf $7$} has a unique $\{\mathfrak{h}_2\}$-saturated Hoffman graph containing it, up to strong isomorphism.
\end{theorem}

In \cite{Cvetkovic.1981}, Cvetkovi\'{c}, Doob and Simi\'{c} showed a similar result for generalized line graphs. We formulate their result in terms of line Hoffman graphs as follows.
\begin{theorem}
Every $\{\mathfrak{h}_2, \mathfrak{h}_3\}$-line Hoffman graph whose slim graph is connected of order at least $7$ has a unique $\{\mathfrak{h}_2, \mathfrak{h}_3\}$-saturated Hoffman graph containing it, up to strong isomorphism.
\end{theorem}

Taniguchi \cite{Taniguchi.2008} showed the following result, although in his paper he used different terminology.
\begin{theorem} \label{tani}
Every $\{\mathfrak{h}_2, \mathfrak{h}_5\}$-line Hoffman graph whose slim graph is connected of order at least $8$ has a unique $\{\mathfrak{h}_2, \mathfrak{h}_5\}$-saturated Hoffman graph containing it, up to strong isomorphism.
\end{theorem}

In \cite{Furuya.2020}, Furuya, Kubota, Taniguchi and Yoshino generalized Theorem \ref{tani}. In \cite{Taniguchi.2012}, Taniguchi showed that if a graph is not the slim graph of a $\{\mathfrak{h}_2, \mathfrak{h}_5\}$-line Hoffman graph but each of its proper induced subgraph is the slim graph of a $\{\mathfrak{h}_2, \mathfrak{h}_5\}$-line Hoffman graph, then this graph is just isomorphic to one of $38$ graphs, found by
computer. In \cite{Kubota.2019}, Kubota, Taniguchi and Yoshino gave more related results.

\subsection{Minimal fat Hoffman graphs}
Let $\mu<0$ be a real number. A Hoffman graph $\mathfrak{h}$ is said to be \emph{$t$-fat-minimal} for $\mu$, if it is $t$-fat, its smallest eigenvalue is less than $\mu$, and each of its proper $t$-fat induced Hoffman subgraph has smallest eigenvalue at least $\mu$. For convenience, a $1$-fat-minimal Hoffman graph for $\mu$ is also said to be \emph{fat-minimal} for $\mu$.

Woo and Neumaier \cite{Woo} determined all the fat-minimal Hoffman graphs for $-1-\sqrt{2}$. By checking their results, one can find that every fat-minimal Hoffman graph for $-1-\sqrt{2}$ has at most $4$ slim vertices.

Later Koolen et al. \cite{kyy3} studied fat-minimal Hoffman graphs for $-3$. They found that every fat-minimal Hoffman graph for $-3$ has at most $10$ slim vertices. Moreover, if a fat-minimal Hoffman graph for $-3$ has a slim vertex with at least $2$ fat neighbors, then it has at most $2$ slim vertices.

Here we introduce an important family of fat Hoffman graphs. Let $H$ be a graph. Let $\mathfrak{q}(H)$ be the fat Hoffman graph with slim graph $H$ and one fat vertex attached to all slim vertices. Then the special matrix of  $\mathfrak{q}(H)$ is $-I-A(\overline{H})$, where $\overline{H}$ is the complement of $H$. Let $\lambda_{0}(\overline{H})$ be the largest eigenvalue of $\overline{H}$. We have:
\begin{equation}\label{eq:eigenvalue}
\lambda_{\min}(\mathfrak{q}(H)) = -1-\lambda_{0}(\overline{H})
\end{equation}
immediately. We will show that if $\mathfrak{q}(H)$ is fat-minimal for $\mu$, then its smallest eigenvalue can not be much smaller than $\mu$. To prove it, some preparation is necessary.

\begin{lemma}[{\cite[Corollary 2.2]{S2015}}]\label{lemste}
Let $G=(V(G),E(G))$ be a connected graph with the largest eigenvalue $\lambda_0(G)$ and the principal eigenvector $\mathbf{v}$, the positive eigenvector of norm $1$ with eigenvalue $\lambda_0(G)$. For each vertex $x$ of $G$, let $G-x$ be the subgraph of $G$ induced on $V(G)-\{x\}$ with the largest eigenvalue $\lambda_0(G-x)$. Then
\begin{equation}\label{radius inequality}
  \frac{1-2\mathbf{v}_x^2}{1-\mathbf{v}_x^2}\lambda_0(G)\leq\lambda_0(G-x)\leq\lambda_0(G),
\end{equation}
where $\mathbf{v}_x$ is the $x$-coordinate of the vector $\mathbf{v}$.
\end{lemma}
This lemma has the following consequence.
\begin{proposition}\label{spectral}
Let $\lambda$ be a positive real number and $G$ a graph of order $n\geq3$ with the largest eigenvalue $\lambda_0(G)$. If $\lambda_0(G) >\lambda$ and for every proper induced subgraph $H$ of $G$, the largest eigenvalue  $\lambda_0(H)$ of $H$ is at most $\lambda$, then $\lambda_0(G) \leq  \frac{n-1}{n-2}\lambda$.
\end{proposition}
\begin{proof}
It is not hard to see that $G$ is connected. Let $\mathbf{v}$ be the principal eigenvector of $G$. Take a vertex $x$ of $G$ such that $\mathbf{v}_x$ is minimal. Then $\mathbf{v}_x \leq \frac{1}{\sqrt{n}}$ as the norm of $\mathbf{x}$ is equal to $1$.  Now apply Lemma \ref{lemste} to this vertex $x$ and we obtain the desired inequality. This completes the proof.
\end{proof}

\begin{lemma}\label{minimal forbidden}
Let $\mu<0$ be a real number and $n\geq3$ a positive integer. Let $\mathfrak{q}(H)$ be a Hoffman graph with $n$ slim vertices. If $\mathfrak{q}(H)$ is fat-minimal for $\mu$, then $\lambda_{\min}(\mathfrak{q}(H))\geq\mu+\frac{1+\mu}{n-2}$.
\end{lemma}
\begin{proof}
As $\mathfrak{q}(H)$ is a fat-minimal Hoffman graph for $\mu$, we have $\lambda_{\min}(\mathfrak{q}(H))<\mu$, and thus $\lambda_{0}(\overline{H})>-1-\mu$ by \eqref{eq:eigenvalue}. Assume $K:=\overline{H'}$ is a proper induced subgraph of $\overline{H}$, where $H'$ is a proper induced subgraph of $H$. Considering the minimality of $\mathfrak{q}(H)$, we have $-1-\lambda_{0}(\overline{H'})=\lambda_{\min}(\mathfrak{q}(H'))\geq\mu$ by \eqref{eq:eigenvalue}. This means $\lambda_{0}(K)=\lambda_{0}(\overline{H'})\leq-1-\mu$. Now the conditions of Proposition \ref{spectral} are satisfied, and we can easily obtain  $\lambda_{0}(\overline{H})\leq\frac{n-1}{n-2}(-1-\mu)$. By using \eqref{eq:eigenvalue} again, we have $\lambda_{\min}(\mathfrak{q}(H))=-1-\lambda_{0}(\overline{H})\geq\mu+\frac{1+\mu}{n-2}$.
\end{proof}

\subsection{Maximal \texorpdfstring{$\mu$}{\mu}-irreducible Hoffman graphs}

A $\mu$-irreducible Hoffman graph is \emph{maximal}, if it is not a proper induced Hoffman subgraph of another $\mu$-irreducible Hoffman graph. Notice that if a $\mu$-irreducible Hoffman graph is maximal, then it is $\mu$-saturated and indecomposable.

Woo and Neumaier \cite{Woo} found that there are exactly $4$ maximal $(-1-\sqrt{2})$-irreducible Hoffman graphs, up to isomorphism, and they are $\mathfrak{h}_2$, $\mathfrak{h}_5$, $\mathfrak{h}_7$ and $\mathfrak{h}_9$ in Figure \ref{fig:-1-sqrt2}.

Let $\tau$ be the golden ratio $\frac{1 + \sqrt{5}}{2}$. In 2014, Munemasa, Sano and Taniguchi \cite{Munemasa.2014b} found that there are exactly $18$ maximal $(-1-\tau)$-irreducible Hoffman graphs, up to isomorphism, and they also gave a list of these $18$ Hoffman graphs.

As for the fat maximal $(-3)$-irreducible Hoffman graphs, we refer to \cite{HJAT} and \cite{Koolen.2018}. To state the main results there, we need to define the special graph of a Hoffman graph. (Signed graphs and switching equivalence will be introduced in Section \ref{sec:signed graph}.)

\begin{definition}[special graph]
The special graph of a Hoffman graph $\mathfrak{h}$ is the signed graph
\[
\mathcal{S}(\mathfrak{h}):=(V(\mathcal{S}(\mathfrak{h})),E^+(\mathcal{S}(\mathfrak{h})),E^-(\mathcal{S}(\mathfrak{h}))),
\]
where $V(\mathcal{S}(\mathfrak{h}))=V_{\mathrm{slim}}(\mathfrak{h})$ and

\begin{align*}
 E^+(\mathcal{S}(\mathfrak{h}))=&\{\{x,y\}\mid x,y\in V_{\mathrm{slim}}(\mathfrak{h}),x\neq y,\{x,y\}\in E(\mathfrak{h}), N_\mathfrak{h}^f(x,y)=\emptyset\},\\
 E^-(\mathcal{S}(\mathfrak{h}))=&\{\{x,y\}\mid x,y\in V_{\mathrm{slim}}(\mathfrak{h}),x\neq y,\{x,y\}\in E(\mathfrak{h}), |N_\mathfrak{h}^f(x,y)|\geq2\}\\
                               \cup&\{\{x,y\}\mid x,y\in V_{\mathrm{slim}}(\mathfrak{h}),x\neq y,\{x,y\}\not\in E(\mathfrak{h}), N_\mathfrak{h}^f(x,y)\neq\emptyset\}.
\end{align*}
\end{definition}
The \emph{special $\varepsilon$-graph} of $\mathfrak{h}$ is the graph $S^\epsilon(\mathfrak{h})=(V_{\mathrm{slim}}(\mathfrak{h}),E^\epsilon(\mathcal{S}(\mathfrak{h})))$ for $\epsilon\in\{+,-\}$.

Let $\mathfrak{h}$ be a fat indecomposable Hoffman graph with $\lambda_{\min}(\mathfrak{h})\geq-3$. It is shown in \cite{HJAT} that, if $\mathfrak{h}$ is not the Hoffman graph \raisebox{-1ex}{\begin{tikzpicture}[scale=0.3]
\tikzstyle{every node}=[draw,circle,fill=black,minimum size=10pt,scale=0.3,
                            inner sep=0pt]

    \draw (-2.1,0) node (1f1) [label=below:$$] {};
    \draw (-1.6,0) node (1f2) [label=below:$$] {};
    \draw (-1.1,0) node (1f3) [label=below:$$] {};

    \tikzstyle{every node}=[draw,circle,fill=black,minimum size=5pt,scale=0.3,
                            inner sep=0pt]

    \draw (-1.6,1) node (1s1) [label=below:$$] {};

    \draw (1f1) -- (1s1) -- (1f2);
    \draw (1f3) -- (1s1);
    \end{tikzpicture}}, then its special graph is connected and the lattice $\Lambda^{\text{red}}(\mathfrak{h},3)$ is either an irreducible root lattice or a sublattice of the
standard lattice. Moreover, if $\mathfrak{h}$ is a fat maximal $(-3)$-irreducible Hoffman graph and the lattice $\Lambda^{\text{red}}(\mathfrak{h},3)$ is a sublattice of the standard lattice, that is $\mathfrak{h}$ is fat maximal $(-3)$-irreducible with an integral representation of norm $3$, then the graph $S^-(\mathfrak{h})$ is connected and is (isomorphic to) the Dynkin diagrams $A_n$, $D_n$ or the extended Dynkin diagram $\hat{A}_n,\hat{D}_n$ for some positive integer
$n$ (see Figure \ref{fig:Dynkin}). Using this result of \cite{HJAT}, Koolen, Li and Yang \cite{Koolen.2018} classified the fat maximal $(-3)$-irreducible Hoffman graphs with an integral representation of norm $3$. As for the fat maximal $(-3)$-irreducible Hoffman graphs with no integral representation of norm $3$, the classification is still open.

\begin{figure}[ht]

    \centering

    \begin{tikzpicture}

    \draw (-1.5,2.5) node {$A_n$};

    \draw (5.5,2.5) node {$D_n$};

    \draw (-1.5,0) node {$\hat{A}_n$};

    \draw (5.5,0) node {$\hat{D}_n$};

    \tikzstyle{every node}=[draw,circle,fill=white,minimum size=4pt, inner sep=0pt]

                            {every label}=[\tiny]

    \draw (-0.5,2.5) node (0) [label=below:$$] {}

        -- ++(0:1cm) node (1) [label=below:$$] {}

        -- ++(0:1cm) node (2) [label=below:$$] {};

    \draw (3.5,2.5) node (4) [label=below:$$] {}

        -- ++(180:1cm) node (3) [label=below:$$] {};

    \draw [dashed] (2) -- (3);

    \draw (6.5,2.5) node (11) [label=below:$$] {}

        -- ++(0:1cm) node (12) [label=below:$$] {}

        -- ++(0:1cm) node (12) [label=below:$$] {};

    \draw (10,2.5) node (13) [label=below:$$] {}

        -- ++(0:1cm) node (14) [label=below:$$] {}

        -- ++(0:1cm) node (15) [label=below:$$] {};

    \draw [dashed] (12) -- (13);

    \draw (14)

        -- ++(90:0.8cm) node (16) [label=below:$$] {};

    \draw (1.5,0.8) node (0) [label=above:$$] {};

    \draw (-0.5,0) node (1) [label=below:$$] {}

        -- ++(0:1cm) node (2) [label=below:$$] {};

    \draw (3.5,0) node (4) [label=below:$$] {}

        -- ++(180:1cm) node (3) [label=below:$$] {};

    \draw [dashed] (2) -- (3);

    \draw (0) -- (1);

    \draw (0) -- (4);

    \draw (7.5,0.8) node (01) [label=above:$$] {};

    \draw (11,0.8) node (02) [label=above:$$] {};

    \draw (6.5,0) node (11) [label=below:$$] {}

        -- ++(0:1cm) node (12) [label=below:$$] {}

        -- ++(0:1cm) node (13) [label=below:$$] {};

    \draw (10,0) node (14) [label=below:$$] {}

        -- ++(0:1cm) node (15) [label=below:$$] {}

        -- ++(0:1cm) node (16) [label=below:$$] {};

    \draw [dashed] (13) -- (14);

    \draw (01) -- (12);

    \draw (02) -- (15);

    \end{tikzpicture}
{\caption{}\label{fig:Dynkin}}
\end{figure}

A related result was shown by Munemasa, Sano and Taniguchi \cite{Munemasa.2014} and Greaves, Koolen, Munemasa, Sano and Taniguchi \cite{Greaves.2015}. In \cite{Munemasa.2014}, Munemasa et al. gave a characterization of special graphs of fat Hoffman graphs with smallest eigenvalue greater than $-3$ which contain a slim vertex having two fat neighbors. In \cite[Theorem 20]{Greaves.2015}, Greaves et al. showed that for a fat Hoffman graph in which every slim vertex has exactly one fat neighbor, it has smallest eigenvalue greater than $-3$ if and only if its special graph is switching equivalent to certain signed graphs.

%
%\subsection{Applications of Hoffman graphs}
%
%In this subsection we discuss some applications of Hoffman graphs in other fields.
%
%In \cite{kubota.2019b}, Kubota, Segawa, Taniguchi and Yoshie study quantum walks in graphs using Hoffman graphs.
%
%In Corollary \ref{bound on se of regular graphs} we gave a result on the smallest eigenvalue on regular graphs which do not contain induced $K_{1,t}$'s. By using Hoffman graphs, one can try to understand  the structure theory of claw-free graphs (i.e. graphs that do not contain induced $K_{1,3}$'s), as developed by Chudnovsky and Seymour \cite{} the eight papers by them. This may help to improve Corollary \ref{bound on se of regular graphs}  for claw-free graphs (with large valency).

\section{Graphs with large minimal valency}\label{sec:associated hoffman graphs}
In this section, we give some results on graphs with fixed smallest eigenvalue and large minimal valency. We start with the associated Hoffman graphs.
\subsection{Associated Hoffman graphs}\label{asho}

In this subsection, we summarize some facts about associated Hoffman graphs and quasi-cliques, which provide some connections between Hoffman graphs and graphs. For more details, we refer to \cite{kky} and \cite{kyy1}.

Let $m$ be a positive integer and $G$ a graph that does not contain $\widetilde{K}_{2m}$ as an induced subgraph, where $\widetilde{K}_{2m}$ is the graph defined before Theorem \ref{Hoff1973}. Let $\mathcal{C}(n)=\{C\mid$ $C$ is a maximal clique of $G$ of order at least $n\}$. Define the relation $\equiv_n^m$ on $\mathcal{C}(n)$ by $C_1\equiv_n^m C_2$ if each vertex $x\in C_1$ has at most $m-1$ non-neighbors in $C_2$ and each vertex $y\in C_2$ has at most $m-1$ non-neighbors in $C_1$. Note that $\equiv_n^m$ is an equivalence relation if $n \geq (m+1)^2$.

Let $[C]_n^m$ denote the equivalence class of $\mathcal{C}(n)$ of $G$ under the equivalence relation $\equiv_n^m$ containing the maximal clique $C$ of $\mathcal{C}(n)$. We define the \emph{quasi-clique} $Q([C]_n^m)$ of $C$ with respect to the pair $(m, n)$ as the subgraph of $G$ induced on the set $\{x\in V(G)\mid$ $x$ has at most $m-1$ non-neighbors in $C\}$. Note that for any $C'\in [C]_n^m$, we have $Q([C']_n^m)=Q([C]_n^m)$ (see \cite[Lemma 3.3]{kky}).

Let $[C_1]_n^m, \dots, [C_r]_n^m$ be the equivalence classes of maximal cliques under $\equiv_n^m$. The \emph{associated Hoffman graph} $\mathfrak{g}=\mathfrak{g}(G,m,n)$ is the Hoffman graph satisfying the following conditions:
\begin{enumerate}
\item $V_{\mathrm{slim}}(\mathfrak{g}) = V(G)$, $V_{\mathrm{fat}}(\mathfrak{g})=\{f_1, f_2, \dots, f_r\}$;
\item the slim graph of $\mathfrak{g}$ equals $G$;
\item for each $i$, the fat vertex $f_i$ is adjacent to exactly all the vertices of $Q([C_i]_n^m)$ for $i=1,2, \dots, r.$
\end{enumerate}

From the above definition of associated Hoffman graphs, we find that for each $i=1,\dots,r$, the quasi-clique $Q([C_i]_n^m)$ of $C_i$ with respect to the pair $(m, n)$ is exactly the quasi-clique $Q_{\mathfrak{g}}(f_i)$ in $\mathfrak{g}$ with respect to the fat vertex $f_i$.

The following result, which is a crucial tool for the study of graphs with fixed smallest eigenvalue and large minimal valency, was shown in \cite[Proposition 4.1]{kky}.

\begin{proposition}\label{assohoff}
Let $G$ be a graph and let $m \geq 2, \phi,\sigma,p \geq 1$ be integers.
There exists a positive integer $n = n(m, \phi, \sigma,  p) \geq (m+1)^2$ such that, for any integer $q \geq n$ and any Hoffman graph $\mathfrak{h}$ with at most $\phi$ fat vertices and at most $\sigma$ slim vertices, the graph  $G(\mathfrak{h}, p)$ is an induced subgraph of $G$, provided that the graph $G$ satisfies the following conditions:
\begin{enumerate}
\item the graph $G$ does not contain  $\widetilde{K}_{2m}$ as an induced subgraph,
\item its associated Hoffman graph $\mathfrak{g} = \mathfrak{g}(G, m, q)$ contains $\mathfrak{h}$ as an induced Hoffman subgraph.
\end{enumerate}
\end{proposition}
 In order to use this proposition well, we need the following definition.
\begin{definition}[$(r,\lambda)$-nice Hoffman graph]
Let $r$ be a positive integer and $\lambda\leq-1$ a real number. A Hoffman graph $\mathfrak{h}$ is $(r,\lambda)$-nice, if it contains no induced Hoffman subgraph with the number of slim vertices at most $r$ and smallest eigenvalue less than $\lambda$.
\end{definition}
\begin{theorem} \label{nicethm}
Let $\lambda\leq -2$ be a real number and $r$ a positive integer. Let $m$ be such that the smallest eigenvalue of $\widetilde{K}_{2m}$ is less than $\lambda$.
Then there exists a positive integer $N= N(\lambda,m,r) \geq (m+1)^2$ such that, for any graph $G$ with smallest eigenvalue at least $\lambda$, the associated Hoffman graph $\mathfrak{g} = \mathfrak{g}(G, m, q)$ is $(r, \lambda)$-nice if $q\geq N$.
\end{theorem}
\begin{proof}
Let $\mathfrak{h}^{(\lceil-\lambda\rceil+1)}$ be the Hoffman graph with one slim vertex adjacent to $\lceil-\lambda\rceil+1$ fat vertices, and let \[\mathfrak{G}=\left\{\mathfrak{h}^{(\lceil-\lambda\rceil+1)}\right\}\cup\left\{\mathfrak{h}\mid \lambda_{\min}(\mathfrak{h})<\lambda, |V_{\mathrm{slim}}(\mathfrak{h})|\leq r, |N_\mathfrak{h}^f(x)|\leq\lceil-\lambda\rceil \text{ for all }x\in V_{\mathrm{slim}}(\mathfrak{h})\right\}\]
be a family of pairwise non-isomorphic Hoffman graphs. It is not hard to see that the family $\mathfrak{G}$ is finite and we may assume $\mathfrak{G}=\{\mathfrak{f}_1,\mathfrak{f}_2,\ldots,\mathfrak{f}_s\}$. As $\lambda_{\min}(\mathfrak{f}_i)<\lambda$ for each $i=1,\ldots,s$, there exist positive integers $p_i$'s such that $\lambda_{\min}(G(\mathfrak{f}_i,p_i))<\lambda$ hold by Theorem \ref{Ostrowski}. Let $p=\max_{1\leq i\leq s}p_i$, $\phi=\max_{1\leq i\leq s}|V_{\mathrm{fat}}(\mathfrak{f}_i)|$ and let $N$ be the positive integer $n(m,\phi,r,p)$ such that Proposition \ref{assohoff} holds. Now for a given graph $G$ with $\lambda_{\min}(G)\geq\lambda$ and an integer $q>N$, we will show that the associated Hoffman graph $\mathfrak{g}(G, m, q)$ is $(r, \lambda)$-nice. Suppose not. Then $\mathfrak{g}(G, m, q)$ contains a Hoffman graph in $\mathfrak{G}$ as an induced Hoffman subgraph. Without loss of generality, we may assume $\mathfrak{f}_1$ is an induced Hoffman subgraph of $\mathfrak{g}(G, m, q)$. Proposition \ref{assohoff} says that under this condition the graph $G$ contains $G(\mathfrak{f}_1,p)$ as an induced subgraph. Now we obtain a contradiction, as $\lambda_{\min}(G(\mathfrak{f}_1,p))<\lambda$ and $\lambda_{\min}(G)\geq\lambda$. Hence the theorem holds.
\end{proof}

\subsection{Graphs with large minimal valency}
In this subsection, we focus on graphs with fixed smallest eigenvalue and large minimal valency.

Let $G$ be a graph. For a given vertex $x$ of $G$, we call the subgraph of $G$ induced on the neighbors of $x$ the \emph{local graph} of $G$ at $x$, and denote it by $\Delta_G(x)$. For convenience, we also denote by $k_{G}(x)$ the valency of $x$ and $\bar{k}(\Delta_G(x))$ the average valency of the graph $\Delta_G(x)$, that is,
\[k_{G}(x)=|V(\Delta_G(x))|,\mathrm{ and } \bar{k}(\Delta_G(x))=\frac{2|E(\Delta_G(x))|}{|V(\Delta_G(x))|}.\]

Let $p$ be a positive integer. A $p$-\emph{plex} is an induced subgraph in which each vertex is adjacent to all but at most $p$ of the vertices. Note that a clique is exactly the same as a $1$-plex.

We define $\mathfrak{G}(t)$ to be the family of pairwise non-isomorphic indecomposable $t$-fat Hoffman graphs with special matrix $(-t-1)$ or $\begin{pmatrix}
J_{s_1}-(t+1)I & -J \\
-J & J_{s_2}-(t+1)I
\end{pmatrix}$ where $1\le s_1,s_2\le t$.

Using associated Hoffman graphs the following five results have been shown.

\begin{theorem}[cf.~{\cite[Theorem 1.2]{kyy1}}]\label{intro1}
Let $t\geq2$ be an integer. Then there exists a positive integer $C_2(t)$ such that, if a graph $G$ satisfies the following conditions:

\begin{enumerate}
\item $k_G(x)>C_2(t)$ holds for all $x \in V(G),$
\item any $(t^2 +1)$-plex containing $x$ has order at most $\frac{k_G(x)-C_2(t)}{t}$ for all $x \in V(G),$
\item $\lambda_{\min} (G) \geq -t-1$,
\end{enumerate}
then $G$ is the slim graph of a $t$-fat $\mathfrak{G}(t)$-line Hoffman graph.
\end{theorem}

\begin{theorem}[cf.~{\cite[Theorem 1.3]{kyy1}}]\label{intro2}
Let $t\geq2$ be an integer. Then there exists a positive integer $C_3(t)$ such that, if a graph $G$ satisfies the following conditions:

\begin{enumerate}
\item $k_G(x)>C_3(t)$ holds for all $x \in V(G),$
\item $\bar{k}(\Delta_G(x)) \leq \frac{k_G(x)-C_3(t)}{t}$ holds for all $x \in V(G)$,
\item $\lambda_{\min} (G) \geq -t-1$,
\end{enumerate}
then $G$ is the slim graph of a $t$-fat $\mathfrak{G}(t)$-line Hoffman graph.
\end{theorem}

In the next two results, we focus on graphs with smallest eigenvalue at least $-3$.

\begin{theorem}[{\cite[Theorem 1.4]{kyy1}}]\label{intro3}
There exists a positive integer $\kappa_2$ such that, if a graph $G$ satisfies the following conditions:
\begin{enumerate}
\item $k_G(x)>\kappa_2$ holds for all $x \in V(G),$
\item any $5$-plex containing $x$ has order at most $k_G(x)-\kappa_2$ for all $x \in V(G),$
\item $\lambda_{\min} (G) \geq -3$,
\end{enumerate}
then $G$ is the slim graph of a $2$-fat \big\{\raisebox{-1ex}{\begin{tikzpicture}[scale=0.3]

\tikzstyle{every node}=[draw,circle,fill=black,minimum size=10pt,scale=0.3,
                            inner sep=0pt]

    \draw (-2.1,0) node (1f1) [label=below:$$] {};
    \draw (-1.6,0) node (1f2) [label=below:$$] {};
    \draw (-1.1,0) node (1f3) [label=below:$$] {};

    \tikzstyle{every node}=[draw,circle,fill=black,minimum size=5pt,scale=0.3,
                            inner sep=0pt]

    \draw (-1.6,1) node (1s1) [label=below:$$] {};

    \draw (1f1) -- (1s1) -- (1f2);
    \draw (1f3) -- (1s1);
    \end{tikzpicture}},\hspace{-0.08cm}
\raisebox{-1ex}{\begin{tikzpicture}[scale=0.3]
\tikzstyle{every node}=[draw,circle,fill=black,minimum size=10pt,scale=0.3,
                            inner sep=0pt]

    \draw (-0.5,0) node (2f1) [label=below:$$] {};
    \draw (0.5,0) node (2f2) [label=below:$$] {};
    \draw (-0.5,1) node (2f3) [label=below:$$] {};
    \draw (0.5,1) node (2f4) [label=below:$$] {};

    \tikzstyle{every node}=[draw,circle,fill=black,minimum size=5pt,scale=0.3,
                            inner sep=0pt]

    \draw (0,0.2) node (2s1) [label=below:$$] {};
    \draw (0.3,0.5) node (2s2) [label=below:$$] {};
    \draw (-0.3,0.5) node (2s3) [label=below:$$] {};
    \draw (0,0.8) node (2s4) [label=below:$$] {};

    \draw (2f1) -- (2s1) -- (2f2) -- (2s2) -- (2f4) -- (2s4) -- (2f3) -- (2s3) -- (2f1);
    \draw (2s1) -- (2s4);
    \draw (2s2) -- (2s3);
    \end{tikzpicture}},\hspace{-0.08cm}
\raisebox{-1ex}{\begin{tikzpicture}[scale=0.3]
\tikzstyle{every node}=[draw,circle,fill=black,minimum size=10pt,scale=0.3,
                            inner sep=0pt]

    \draw (1.5,0) node (3f1) [label=below:$$] {};
    \draw (1.5,1) node (3f2) [label=below:$$] {};

    \tikzstyle{every node}=[draw,circle,fill=black,minimum size=5pt,scale=0.3,
                            inner sep=0pt]

    \draw (1,0.5) node (3s1) [label=below:$$] {};
    \draw (2,0.5) node (3s2) [label=below:$$] {};

    \draw (3s1) -- (3s2);
    \draw (3f1) -- (3s1) -- (3f2) -- (3s2) -- (3f1);
    \end{tikzpicture}}\big\}-line Hoffman graph.
\end{theorem}

\begin{theorem}[{\cite[Theorem 1.5]{kyy1}}]\label{intro4}
There exists a positive integer $\kappa_3$ such that, if a graph $G$ satisfies the following conditions:
\begin{enumerate}
\item $k_G(x)>\kappa_3$ holds for all $x \in V(G),$
\item $\bar{k}(\Delta_G(x))\leq k_G(x)-\kappa_3$ holds for all $x \in V(G)$,
\item $\lambda_{\min} (G) \geq -3$,
\end{enumerate}
then $G$ is the slim graph of a $2$-fat \big\{\raisebox{-1ex}{\begin{tikzpicture}[scale=0.3]

\tikzstyle{every node}=[draw,circle,fill=black,minimum size=10pt,scale=0.3,
                            inner sep=0pt]

    \draw (-2.1,0) node (1f1) [label=below:$$] {};
    \draw (-1.6,0) node (1f2) [label=below:$$] {};
    \draw (-1.1,0) node (1f3) [label=below:$$] {};

    \tikzstyle{every node}=[draw,circle,fill=black,minimum size=5pt,scale=0.3,
                            inner sep=0pt]

    \draw (-1.6,1) node (1s1) [label=below:$$] {};

    \draw (1f1) -- (1s1) -- (1f2);
    \draw (1f3) -- (1s1);
    \end{tikzpicture}},\hspace{-0.08cm}
\raisebox{-1ex}{\begin{tikzpicture}[scale=0.3]
\tikzstyle{every node}=[draw,circle,fill=black,minimum size=10pt,scale=0.3,
                            inner sep=0pt]

    \draw (-0.5,0) node (2f1) [label=below:$$] {};
    \draw (0.5,0) node (2f2) [label=below:$$] {};
    \draw (-0.5,1) node (2f3) [label=below:$$] {};
    \draw (0.5,1) node (2f4) [label=below:$$] {};

    \tikzstyle{every node}=[draw,circle,fill=black,minimum size=5pt,scale=0.3,
                            inner sep=0pt]

    \draw (0,0.2) node (2s1) [label=below:$$] {};
    \draw (0.3,0.5) node (2s2) [label=below:$$] {};
    \draw (-0.3,0.5) node (2s3) [label=below:$$] {};
    \draw (0,0.8) node (2s4) [label=below:$$] {};

    \draw (2f1) -- (2s1) -- (2f2) -- (2s2) -- (2f4) -- (2s4) -- (2f3) -- (2s3) -- (2f1);
    \draw (2s1) -- (2s4);
    \draw (2s2) -- (2s3);
    \end{tikzpicture}},\hspace{-0.08cm}
\raisebox{-1ex}{\begin{tikzpicture}[scale=0.3]
\tikzstyle{every node}=[draw,circle,fill=black,minimum size=10pt,scale=0.3,
                            inner sep=0pt]

    \draw (1.5,0) node (3f1) [label=below:$$] {};
    \draw (1.5,1) node (3f2) [label=below:$$] {};

    \tikzstyle{every node}=[draw,circle,fill=black,minimum size=5pt,scale=0.3,
                            inner sep=0pt]

    \draw (1,0.5) node (3s1) [label=below:$$] {};
    \draw (2,0.5) node (3s2) [label=below:$$] {};

    \draw (3s1) -- (3s2);
    \draw (3f1) -- (3s1) -- (3f2) -- (3s2) -- (3f1);
    \end{tikzpicture}}\big\}-line Hoffman graph.
\end{theorem}

The following result is important in the proof of Theorem \ref{mthmkyy3}.
\begin{theorem}[{\cite[Theorem 5.3]{kyy3}}]\label{kyy3thm}
There exists a positive integer $\kappa_4$ such that, if $G$ is a graph with smallest eigenvalue at least $-3$ and minimal valency at least $\kappa_4$, then $G$ is the slim graph of a fat Hoffman graph with smallest eigenvalue at least $-3$.
\end{theorem}

We now show how to obtain Theorem \ref{mthmkyy3} from Theorem \ref{kyy3thm}. Assume that $G$ is the slim graph of the fat Hoffman graph $\mathfrak{h}$ with smallest eigenvalue at least $-3$. Without loss of generality, we may assume $-3\leq\lambda_{\min}(G)<-2$ by Theorem \ref{thmcam}. Following from Lemma \ref{rela}, we find that $\mathfrak{h}$ has a reduced representation of norm $3$ and denote it by $\psi$. Then for each vertex $x$ of $G$, we let $N_x:=\psi(x)+\sum\limits_{f\sim x,f\in V_{\mathrm{fat}}(\mathfrak{h})}\mathbf{e}_f$, where $\{\mathbf{e}_f\mid f\in V_{\mathrm{fat}}(\mathfrak{h})\}$ is a set of orthonormal integral vectors which are orthogonal to all of the vectors in the set $\{\psi(y)\mid y\in V(G)\}$. It is not hard to check that the matrix $N$ with $N_x$'s as its columns satisfies the equation $A(G)+3I=N^TN$. Thus \[\Lambda(G)=\langle N_x\mid x\in V(G)\rangle_{\mathbb{Z}}=\langle \psi(x)+\sum_{\substack{f\sim x,\\f\in V_{\mathrm{fat}}(\mathfrak{h})}}\mathbf{e}_f\mid x\in V(G)\rangle_{\mathbb{Z}}.\]
This implies that the integrability of the integral lattice $\Lambda(G)$ depends on the integrability of the integral lattice $\Lambda^{\text{red}}(\mathfrak{h},3):=\langle \psi(x)\mid x\in V(G)\rangle_{\mathbb{Z}}$. Note that for each $x$, the norm of the vector $\psi(x)$ is at most $2$, as $\mathfrak{h}$ is fat. Hence the integral lattice $\Lambda^{\text{red}}(\mathfrak{h},3)$ is the direct sum of $E_i$'s, $D_i$'s, $A_i$'s and $\mathbb{Z}^q$ where $q$ is a non-negative integer. As each of these lattices is
$s$-integrable (see \cite[Corollary 23]{CS}) for any $s\geq2$, Theorem \ref{kyy3thm} holds.

\section{Distance-regular graphs}\label{sec:drg}

In this section, we give some results on distance-regular graphs that have fixed smallest eigenvalue or whose local graphs have fixed smallest eigenvalue.  We start with some definitions.
\subsection{Definitions}
Let $G$ be a $k$-regular graph of order $n$. If every pair of adjacent vertices of $G$ has exactly $a$ common neighbors, and every pair of distinct and non-adjacent vertices of $G$ has exactly $c$ common neighbors, then $G$ is called \emph{strongly regular} with parameters $(n, k, a, c)$.

Let $G$ be a connected graph with diameter $D$. For each vertex $x$ of $G$, denote by $G_i(x)$ the set of vertices at distance precisely $i$ from $x$, where $0\leq i\leq D$. For any two vertices $x$ and $y$, denote by $d(x,y)$ the distance between $x$ and $y$.
The graph $G$ is called \emph{distance-regular} if there are positive integers \[b_0,b_1,\ldots,b_{D-1},c_1=1,c_2,\ldots,c_D\] such that for any two vertices $x$ and $y$ with $d(x,y)=i$, $|G_{i-1}(x)\cap G_1(y)|=c_i$ ($1\leq i\leq D$) and
$|G_{i+1}(x)\cap G_1(y)|=b_i$ ($0\leq i\leq D-1$). Set $b_D=c_0=0$. The numbers $b_i,c_i$ and $a_i:=b_0-b_i-c_i$ ($0\leq i\leq D$) are called the \emph{intersection numbers} of $G$. The array $\iota(G):=\{b_0,b_1,\ldots,b_{D-1};c_1,c_2,\ldots,c_D\}$ is called the \emph{intersection array} of $G$. Note that a distance-regular graph of diameter $2$ and order $n$ is a strongly regular graph with parameters $(n, b_0, a_1, c_2)$.

Let $G$ be a distance-regular graph with valency $k$ and smallest eigenvalue $\lambda_{\min}(G)$.  As observed by Godsil, the order of a clique $C$ in
$G$ satisfies the Delsarte bound, that is,
$$ |V(C)|\leq 1+ \frac{k}{-\lambda_{\min}(G)},$$
and we say that $C$ is a \emph{Delsarte clique} if its order equals $1 + \frac{k}{-\lambda_{\min}(G)}$.
We say that $G$ is \emph{geometric} if there exists a family $\mathcal{F}$ of Delsarte cliques in $G$ such that each edge of $G$ lies in exactly one of the Delsarte cliques of $\mathcal{F}$. Although the family $\mathcal{F}$ of Delsarte cliques in geometric distance-regular graphs is not always unique, usually it is.

\subsection{Strongly regular graphs with fixed smallest eigenvalue}\label{sec:strongly regular graphs}

It is well-known that a connected strongly regular graph such that its complement is disconnected is complete multipartite (see \cite[Lemma 10.1.1]{GD01}). So in this subsection, we will deal with connected strongly regular graphs whose complements are also connected, that is, the so-called \emph{primitive} strongly regular graphs.

We start this subsection with two important families of geometric primitive strongly regular graphs.

{\bf Family 1:} A \emph{Steiner system} $S(2,t,v)$ is a pair $(\mathcal{P},\mathcal{B})$, where $\mathcal{P}$ is a
$v$-element set called a \emph{point set} and $\mathcal{B}$ is a family of $t$-element subsets of $\mathcal{P}$ called a \emph{block set} such that each $2$-element subset of $\mathcal{P}$ is contained in exactly one block. Take a Steiner system $S(2,t,v)$. Construct a graph as follows: its vertex set is the block set of this Steiner system, and two blocks are adjacent whenever they intersect in one point. This gives a geometric strongly regular graph with parameters $(\frac{v(v-1)}{t(t-1)},\frac{t(v-t)}{t-1},(t-1)^2+\frac{v-1}{t-1}-2, t^2)$ (see \cite[p.~5]{Cioaba.2014}).
We call this graph the \emph{block graph} of the Steiner system and call this family the \emph{Steiner family}.

{\bf Family 2:} An \emph{orthogonal array} with parameters $t$ and $v$ is a $t \times v^2$ array with entries in $\{ 1, 2, \ldots, v\}$ such that the $v^2$ ordered pairs in any two distinct rows are all different. We denote an orthogonal array with these parameters by $OA(v,t)$. Note that an $OA(v, 3)$ is equivalent to a Latin square. Take an $OA(v,t)$ orthogonal array $\mathcal O$. Construct a graph as follows: its vertex set is the set of columns of $\mathcal O$, and two columns are adjacent whenever they have the same entries in (exactly) one position. This gives a geometric strongly regular graph with parameters $(v^2,(v-1)t,v-2+(t-1)(t-2),t(t-1))$ (see \cite[Theorem 10.4.2]{GD01}). We call this family the \emph{Latin square family}.

Neumaier \cite{Neumaier.1979} showed the following theorem for strongly regular graphs with fixed smallest eigenvalue.

\begin{theorem}[cf.~{\cite[Theorem 4.7]{Neumaier.1979}}]\label{neuthm}
Let $G$ be a primitive strongly regular graph with parameters $(n,k,a,c)$ and smallest eigenvalue $-\lambda$, where $\lambda\geq2$ is an integer. If
$(\lambda+1)(a+1)-k>(c-1)\binom{\lambda+1}{2}$, then one of the following holds:
\begin{enumerate}
\item $G$ is in the Latin square family with parameters $((s+1)^2,s\lambda,s-1+(\lambda-1)(\lambda-2),\lambda(\lambda-1))$, where $s$ is a positive integer;
\item $G$ is in the Steiner family with parameters $(n,s\lambda,s-1+(\lambda-1)^2,\lambda^2)$, where $s$ is a positive integer and $n=(s+1)(s(\lambda-1)+\lambda)/\lambda$.
\end{enumerate}
\end{theorem}

He showed this theorem in two steps. The first step was to show the following result.

\begin{theorem}[cf.{\cite[Theorem 4.6]{Neumaier.1979}}]\label{srggeometric2}
Let $G$ be a primitive strongly regular graph with parameters $(n,k,a,c)$ and smallest eigenvalue $-\lambda$, where $\lambda\geq2$ is an integer. If $(\lambda+1)(a+1)-k>(c-1)\binom{\lambda+1}{2}$, then $G$ is geometric.
\end{theorem}

As step two, he determined the geometric strongly regular graphs with parameters $(n,k,a,c)$ and smallest eigenvalue $-\lambda$, where $\lambda\geq2$ is an integer, such that
$(\lambda+1)(a+1)-k>(c-1)\binom{\lambda+1}{2}$.

A non-complete connected regular graph is strongly regular if and only if it has exactly $3$ distinct eigenvalues (see \cite[Lemma 10.2.1]{GD01}). Now we discuss some results on connected non-regular graphs with $3$ distinct eigenvalues.

In \cite{vanDam.1998}, Van Dam determined the connected non-regular graphs with $3$ distinct eigenvalues with smallest eigenvalue at least $-2$.

In \cite{Cheng.2018}, Cheng, Greaves and Koolen determined the connected non-regular graphs with $3$ distinct eigenvalues with second largest eigenvalue at most $1$.

Motivated by Neumaier's theorem, that is Theorem \ref{neuthm}, Cheng, Gavrilyuk, Greaves, and Koolen \cite{Cheng.2016} showed that for any $\lambda\geq 2$, there exists a constant $n_1(\lambda)$
such that any connected non-bipartite biregular graph, with exactly $3$ distinct eigenvalues $\lambda_0 > \lambda_1 > \lambda_2$ satisfying $\lambda_1 \leq \lambda$, has at most $n_1(\lambda)$ vertices. In
the same paper, Cheng et al. asked whether the conclusion holds if the condition $\lambda_1\leq \lambda$ is replaced by $\lambda_2 \geq -\lambda$. Cheng and Koolen \cite{Cheng.2017} answered this question in the affirmative. This latter paper used the ideas of Hoffman, without actually introducing Hoffman graphs.

\subsection{Distance-regular graphs with fixed smallest eigenvalue}

There are four important families of geometric distance-regular graphs: the Hamming graphs, the Johnson graphs, the Grassmann graphs and the bilinear forms
graphs. One of the reasons that they are important is that for fixed diameter there are infinitely many graphs in each of these four families.

Let $D\geq1$ and $q\geq2$ be integers. Let $X$ be a finite set of size $q$. The \emph{Hamming graph} $H(D,q)$ is the graph with the vertex set $X^D:=\prod_{i=1}^DX$ (the Cartesian product of $D$ copies of $X$), where two vertices are adjacent whenever they differ in precisely one coordinate. In particular, if $D=2$, then a Hamming graph $H(2,q)$ is the line graph of a complete bipartite graph $K_{q,q}$. We also call a Hamming graph $H(2,q)$ a \emph{($q\times q$)-grid}.

Let $v,D$ be integers. Let $X$ be a finite set of size $v$. The \emph{Johnson graph} $J(v,D)$ is the graph with vertex set $\binom{X}{D}$, the set of all $D$-element subsets of $X$, where two vertices are adjacent whenever they intersect in precisely $D-1$ elements. In particular, if $D=2$, then a Johnson graph $J(v,2)$ is the line graph of
a complete graph $K_v$. We also call a Johnson graph $J(v,2)$ a \emph{triangular graph} and denote it by $T(v)$. (Since the Johnson graph $J(v,D)$ is isomorphic to the Johnson graph $J(v,v-D)$, we always assume $v\geq2D$.)

Let $v,D$ be integers such that $v\geq D$ and $q$ a prime power. Let $\mathbb{F}$ be a finite field of order $q$ and $V$ a $v$-dimensional vector space over $\mathbb{F}$. The \emph{Grassmann graph} $J_q(v,D)$ is the graph with the set of all $D$-dimensional subspaces of $V$ as the vertex set, where two vertices are adjacent whenever their intersection is a $(D-1)$-dimensional subspace of $V$. (Since the Grassmann graph $J_q(v,D)$ is isomorphic to the Grassmann graph $J_q(v,v-D)$, we will also always assume $v\geq2D$.)

Let $D,e$ be integers such that $e\geq D$ and $q$ a prime power. Let $\mathbb{F}$ be a finite field of order $q$. The \emph{bilinear forms graph} $Bil(D\times e,q)$ is the graph with the set of all $D\times e$ matrices over $\mathbb{F}$ as the vertex set, where two vertices are adjacent whenever the rank of their difference equals one.

Koolen and Bang \cite{Koolen.2010} generalized Theorem \ref{srggeometric2} to the class of distance-regular graphs under the extra condition that $c_2 \geq 2$, by showing:

\begin{theorem}[cf.~{\cite[Theorem 1.3]{Koolen.2010}}]\label{thm:nongeometric}
For given $\lambda \geq 2$, there are only finitely many non-geometric distance-regular graphs with both valency and diameter at least $3$, $c_2 \geq 2$ and smallest eigenvalue at least $-\lambda$.
\end{theorem}

Now we are going to show that in Theorem \ref{thm:nongeometric}, the condition $c_2\geq2$ can be removed. Suppose that $G$ is distance-regular with $c_2=1$. It is not hard to see that $G$ has no induced $K_{2,1,1}$'s, and thus $G$ is locally the disjoint union of $(a_1 +1)$-cliques, that is, the local graph of $G$ at each vertex is the disjoint union of $(a_1 +1)$-cliques (see \cite[Propositon 1.2.1]{Bang.2018}).
We say that a distance-regular graph is of \emph{order} $(s,t)$, if it is locally the disjoint union of $t+1$ $K_s$'s. The next lemma gives a sufficient condition for a distance-regular graph of order $(s,t)$ to be geometric.

\begin{lemma}[{cf.~\cite[Proposition 3.1.6]{Suzuki.1999}},{\cite[Corollary 7.7]{vanDam.2016}}]\label{geomlemma}
Let $G$ be a distance-regular graph of order $(s, t)$ with diameter $D \geq 2$. If $s>t$, then $G$ is geometric.
\end{lemma}

Lemma \ref{geomlemma} implies that for a non-geometric distance-regular graph of order $(s, t)$ with smallest eigenvalue at least $-\lambda$, its valency $k$ satisfies $k=s(t+1)\leq t(t+1) \leq (\lambda^2-1)\lambda^2$, as it contains a $(t+1)$-claw $K_{1,t+1}$ with smallest eigenvalue $-\sqrt{t+1}$ as an induced subgraph, which implies $-\sqrt{t+1}\geq -\lambda$ by Lemma \ref{interelacing}. Now using the fact that the Bannai-Ito conjecture, that is, for fixed $k\geq3$ there are only finitely many distance-regular graphs with valency $k$, is true (see \cite[Theorem 1.1]{Bang.2015}), we find that there are only finitely many non-geometric distance-regular graphs with both valency and diameter at least $3$, $c_2=1$ and smallest
eigenvalue at least $-\lambda$. Hence we obtain:

\begin{theorem}\label{thm:nongeometricnew}
For given $\lambda \geq 2$, there are only finitely many non-geometric
distance-regular graphs with both valency and diameter at least $3$ and smallest
eigenvalue at least $-\lambda$.
\end{theorem}

%
%
%\begin{theorem}[{\cite[Theorem 1.2]{Koolen.2010}}]\label{thm:nongeometric}
%For given integers $m\geq 2$ and $D\geq2$, there are only finitely many coconnected non-geometric
%distance-regular graphs with smallest
%eigenvalue at least $-m$ and diameter $D$.
%\end{theorem}

\begin{remark}
\begin{enumerate}
\item Note that the (general) result was known for $\lambda=2$, see \cite[Theorem 3.12.4,~Theorem 4.2.16]{bcn}.
\item Godsil \cite{Godsil.1993} gave a special case for this result for antipodal distance-regular graphs of diameter $3$.
\end{enumerate}
\end{remark}

Koolen and Bang \cite{Koolen.2010} gave the following two conjectures for geometric distance-regular graphs.

\begin{conjecture}[{\cite[Conjecture 7.4]{Koolen.2010}}]\label{koolenbang} For a fixed integer $\lambda\geq2$, any geometric distance-regular graph with smallest eigenvalue $-\lambda$,
diameter $D\geq3$ and $c_2 \geq 2$ either is a Hamming graph, a Johnson graph, a Grassmann graph, a bilinear forms
graph, or has the number of vertices bounded above by a function of $\lambda$.\end{conjecture}
\begin{conjecture}[{\cite[Conjecture 7.5]{Koolen.2010}}]\label{koolenbang2} For a fixed integer $\lambda\geq2$, the diameter of a geometric distance-regular graph with smallest
eigenvalue $-\lambda$ and valency at least $3$ is bounded above by a function of $\lambda$.
\end{conjecture}
Note that Conjecture \ref{koolenbang} can be seen as a generalization of Theorem \ref{neuthm} for the class of distance-regular graphs.
Now we discuss recent progress on Conjecture \ref{koolenbang2}. Bang \cite{Bang.2018} showed the following result:
\begin{proposition}[{cf.~\cite[Theorem 1.1]{Bang.2018}}]\label{bang}
Fix an integer $\lambda \geq2$. Suppose that $G$ is a geometric distance-regular graph with diameter $D \geq 2$ and smallest eigenvalue $-\lambda$. If $G$ contains an induced subgraph $K_{2,1,1}$, then $D \leq \lambda$.
\end{proposition}

For a distance-regular graph $G$ with intersection array $\{b_0,b_1,\ldots,b_{D-1};c_1,c_2,\ldots,c_D\}$, define the head $h := h(G)= |\{i \mid (c_i, a_i, b_i) = (c_1, a_1, b_1)\}|$.
The following is known for geometric distance-regular graphs with order $(s, t)$ and is due to Suzuki \cite{Suzuki.1999}.

\begin{proposition}[{cf.~\cite[Proposition 3.1.6]{Suzuki.1999}},{\cite[Corollary 7.7]{vanDam.2016}}]\label{suz}
Let $G$ be a distance-regular graph of order $(s, t)$ with head $h$ and diameter $D \geq 2$. If $s>t$, then $D\leq t(h+1)+1$.
\end{proposition}

Note that the smallest eigenvalue of a geometric distance-regular graph of order $(s, t)$ is equal to $-t-1$. %(see \cite[Theorem 3.1.4]{Suzuki.1999}).
Combining Propositions \ref{bang} and \ref{suz}, we obtain:
\begin{corollary}
Let $\lambda\geq2$ be an integer and $G$ a geometric distance-regular graph with head $h$, diameter $D \geq 2$ and smallest eigenvalue $-\lambda$. Then $D \leq (\lambda-1)(h+1) +1$.
\end{corollary}
This means that in order to solve Conjecture \ref{koolenbang2}, we only need to show that the head of a geometric distance-regular graph with smallest eigenvalue $-\lambda$ is bounded above by a function of $\lambda$.

\subsection{Characterizations of distance-regular graphs}

Note that Bannai's problem asks to classify the $Q$-polynomial distance-regular graphs with large diameter. (See Appendix \ref{appendix} for the definition of a $Q$-polynomial distance-regular graph.) One part of this problem is to characterize the known infinite families as a
distance-regular graph. We say that a distance-regular graph has \emph{classical parameters} $(D,b,\alpha,\beta)$, if its diameter is $D$ and its intersection array is given as follows:
\begin{align*}
  b_i=&(\genfrac{[}{]}{0pt}{}{D}{1}-\genfrac{[}{]}{0pt}{}{i}{1})(\beta-\alpha\genfrac{[}{]}{0pt}{}{i}{1}), \\
  c_i=&\genfrac{[}{]}{0pt}{}{i}{1}(1+\alpha\genfrac{[}{]}{0pt}{}{i-1}{1})
\end{align*}
($i=0,1,\ldots,D$), where
$$
\genfrac{[}{]}{0pt}{}{i}{j}=
\left\{
\begin{array}{ll}
  \prod\limits_{\ell=0}^{j-1}\frac{i-\ell}{j-\ell}& \text{ if }b=1,\\
  \prod\limits_{\ell=0}^{j-1}\frac{b^i-b^\ell}{b^j-b^\ell} & \text{ if }b\neq1.
\end{array}
\right.
$$
An important subproblem of Bannai's problem is to classify the distance-regular graphs with classical parameters $(D, b, \alpha,
\beta)$, as every distance-regular graph with classical parameters is $Q$-polynomial (see \cite[Corollary 8.4.1]{bcn}). Note that all the known infinite families of distance-regular graphs with unbounded diameter are either classical or closely  related to classical distance-regular graphs. If a distance-regular graph $G$ has classical parameters $(D, b, \alpha, \beta)$ where $D\geq3$, then $b$ is an integer equals neither $0$ nor $-1$ and if $b$ is positive, then the second largest
eigenvalue $\lambda_1$ of $G$ satisfies $\lambda_1=\frac{b_1}{b} -1$ (see \cite[Proposition 6.2.1]{bcn} and \cite[Corollary 8.4.2]{bcn}).
In Lemma \ref{laminloc} we will relate the smallest eigenvalue of any local graph of a distance-regular graph with $\frac{b_1}{\lambda_1 +1}$, where $\lambda_1$ is the second largest eigenvalue of this distance-regular graph.
For the rest of this subsection, we discuss the characterization of distance-regular graphs with second largest eigenvalue $\lambda_1 = b_1 -1$,
the Grassmann graphs and the bilinear forms graphs.
For more details and also for more characterizations, see \cite{Gavrilyuk.2019} and \cite[Section 9]{vanDam.2016}. For more information on the known families of distance-regular graphs with unbounded diameter, see \cite[Chapter 9]{bcn}.

\vspace{3mm}
\noindent
{\bf Characterization of distance-regular graphs with second largest eigenvalue $\lambda_1 = b_1 -1$}

\vspace{3mm}

It is known that, if a distance-regular graph contains an induced quadrangle,
then the second largest eigenvalue $\lambda_1$ is at most $b_1 -1$ (see \cite[Proposition 4.4.9]{bcn}). Distance-regular graphs that have second largest eigenvalue $b_1 -1$ include the Hamming graphs and the Johnson graphs.
In the next result we characterize the distance-regular graphs with second largest eigenvalue $\lambda_1=b_1 -1$. This result is due to Terwilliger and Neumaier, independently.

\begin{theorem}[cf.~{\cite[Thoerem 4.4.11]{bcn}}]
Let $G$ be a distance-regular graph with second largest eigenvalue $\lambda_1 = b_1 -1$. Then at least one of the following holds:
\begin{enumerate}
\item $G$ is a strongly regular graphs with smallest eigenvalue $-2$;
\item $c_2 =1$;
\item $c_2 = 2$ and $G$ is a Hamming graph $H(D, q)$, a Doob graph {\rm (see \cite[p.~262]{bcn})}, the Conway-Smith graph with intersection array $\{10, 6, 4, 1; 1, 2, 6, 10\}$ {\rm(see \cite[p.~399]{bcn})} or the Doro graph with intersection array $\{ 10, 6, 4; 1, 2, 5\}$ {\rm(see \cite[Chapter 12.1]{bcn})};
\item $c_2 =4 $ and $G$ is a Johnson graph $J(v, D)$ where $v \geq 2D$;
\item $c_2 = 6$ and $G$ is a halved cube {\rm(see \cite[p.~264]{bcn})};
\item $c_2 =10$ and $G$ is the Gosset graph with intersection array $\{27, 10, 1; 1, 10, 27\}$ {\rm(see \cite[p.~103]{bcn})}.
\end{enumerate}
\end{theorem}

Note that this theorem also classifies the distance-regular graphs with classical parameters $(D, 1, \alpha, \beta)$ (see \cite[Theorem 6.1.1]{bcn}).

As corollaries we have the following characterizations of the Hamming graphs and the Johnson graphs. The characterization of the Hamming graphs is due to Egawa and the characterization of the Johnson graphs is due to Terwilliger.

\begin{theorem}[cf.~{\cite[Corollary 9.2.5]{bcn}}]
Let $G$ be a distance-regular graph with the same intersection array as a Hamming graph $H(D, q)$. Then $G$ is the Hamming graph $H(D, q)$ or, if $q=4$, a Doob graph.
\end{theorem}

\begin{theorem}[cf.~{\cite[Corollary 1.2]{Terwilliger.1986}}]
Let $G$ be a distance-regular graph with the same intersection array as a Johnson graph $J(v, D)$, where $v \geq 2D$. Then $G$ is the Johnson graph $J(v, D)$, or if $(v, D) = (8, 2)$, a Chang graph.
\end{theorem}

\vspace{3mm}
\noindent
{\bf Characterization of the Grassmann graphs}

\vspace{3mm}
Metsch \cite{Metsch.1995}, building on work of many people, characterized the Grassmann graphs as follows.
\begin{theorem}[{\cite[Corollary 1.2]{Metsch.1995}}]
Let $G$ be a distance-regular graph  with the same intersection array as
a Grassmann graph $J_q(v,D)$, where $v \geq 2D\geq 6$ are integers and $q \geq 2$ is a prime power. Then $G$ is the Grassmann graph $J_q(v,D)$  if $v \geq \max\{2D+2, 2D+6-q\}$.
\end{theorem}

His approach is to find large cliques, called grand cliques and show that each edge lies in such a unique grand clique.

Gavrilyuk and Koolen \cite{Gavrilyuk.2018} characterized the Grassmann graph $J_q(2D,D)$ as follows.
\begin{theorem}
Let $G$ be a distance-regular graph with the same intersection array as
a Grassmann graph $J_q(2D,D)$, where $D \geq 3$ is an integer and $q \geq 2$ is a prime power. Then $G$ is the Grassmann graph $J_q(2D,D)$  if $D \geq 9$.
\end{theorem}

They first showed that the local graph of such a graph
must have the same spectrum as the $q$-clique extension of the $(\frac{q^D-1}{q-1}\times \frac{q^D-1}{q-1})$-grid. Furthermore, using the $Q$-polynomial property they showed that the local graph at any vertex satisfies that every pair of distinct non-adjacent vertices has the same number of common neighbors (we call graphs with this property co-edge regular graphs and study them in the next section). Then following from Theorem \ref{main-result}, they had that the local graph is really the $q$-clique extension of the $(\frac{q^D-1}{q-1}\times \frac{q^D-1}{q-1})$-grid, if $D\geq 9$. Now building on work of Numata, Cooperstein and Cohen (see \cite[Theorem 9.3.8]{bcn}), they completed their proof.

 Note that the situation for distance-regular graphs with the same parameters as a Grassmann graph $J_q( 2D+1, D)$ is very different, as Van Dam and Koolen \cite{vanDam.2005} constructed
 a distance-regular graph $\tilde{J}_q( 2D+1, D)$ with the same parameters as $J_q(2D+1,D)$, where $q$ is a prime power and $D \geq 2$ is an integer. For these
 graphs, not every edge lies in a maximum clique. This shows that the method used by Metsch can not work in this case. Whether there is a geometric argument for this case is not clear on this
 moment. Note that Munemasa and Tonchev \cite{Munemasa.2011} showed that the block graph of the design constructed by Jungnickel and Tonchev \cite{Jungnickel.2009} is (isomorphic to) the graph
 $\tilde{J}_q(2D+1, D)$.

\vspace{3mm}
\noindent
{\bf Characterization of the bilinear forms graphs.}

\vspace{3mm}
Metsch \cite{Metsch.1999}, again building on work of many people, characterized the bilinear forms graphs as follows.
\begin{theorem}[{\cite[Corollary 1.2]{Metsch.1999}}]
Let $G$ be a distance-regular graph with the same intersection array as
a bilinear forms graph $Bil(D \times e, q)$, where $e \geq D \geq 3$ are integers and $q\geq 2$ is a prime power.
If $q = 2$ and $e \geq
D + 4$ or $q \geq 3$ and
$e \geq D + 3$, then $G$ is the bilinear forms graph $Bil(D \times e, q)$.
\end{theorem}

His approach is the same as for the Grassmann graphs.

Gavrilyuk and Koolen \cite{Gavrilyuk.2019b} characterized the bilinear forms graph $Bil(D \times D, 2)$ as follows.
\begin{theorem}[{\cite[Theorem 1.3]{Gavrilyuk.2019b}}]
Let $G$ be a distance-regular graph  with the same intersection array as
a bilinear forms graph $Bil(D\times D, 2)$, where $D \geq 3$ is an integer.
Then $G$ is the bilinear forms graph $Bil(D \times D, 2)$.
\end{theorem}

They first showed that the local graph of such a graph must be a $((2^D-1) \times (2^D -1))$-grid by showing that the local graph must have the same spectrum as the
$((2^D-1) \times (2^D -1))$-grid. Then they used the cliques of order $2^D$ to construct a geometry and using this geometry they were able to show that the graph must be the bilinear
forms graph $Bil(D \times D, 2)$.

\subsection{Graphs cospectral to a distance-regular graph}
In this subsection, we give some results on graphs cospectral to a distance-regular graph.

We start with the Hamming graphs. Bang, Van Dam and Koolen \cite{Bang.2008} showed the following results.

 \begin{proposition}[{\cite[Theorem 3.4]{Bang.2008}}]
 Let $q$ and $D\geq2$ be positive integers. Let $2q>D^4+2D^3+2D^2-5D-4$. Then any graph that is cospectral to $H(D,q)$ is locally the disjoint union of $D$ cliques of order $q-1$.
 \end{proposition}

 They employed  this result to show:
 \begin{proposition}[{\cite[Theorem 4.5]{Bang.2008}}]
 Let $q \geq 36$. Then the Hamming graph $H(3,q)$ is determined by its spectrum.
 \end{proposition}

 Using Theorem \ref{intro2}, Koolen Yang and Yang showed the following weaker result.

 \begin{theorem}[{cf.~\cite[Theorem 1.6]{kyy1}}] \label{Hamming} There exists a positive integer $q^\prime$ such that for each $q \geq q^\prime$, any graph that is cospectral to the Hamming graph $H(3,q)$ is the slim graph of a $3$-fat \big\{\raisebox{-1ex}{\begin{tikzpicture}[scale=0.3]

\tikzstyle{every node}=[draw,circle,fill=black,minimum size=10pt,scale=0.3,
                            inner sep=0pt]

    \draw (-2.1,0) node (1f1) [label=below:$$] {};
    \draw (-1.6,0) node (1f2) [label=below:$$] {};
    \draw (-1.1,0) node (1f3) [label=below:$$] {};

    \tikzstyle{every node}=[draw,circle,fill=black,minimum size=5pt,scale=0.3,
                            inner sep=0pt]

    \draw (-1.6,1) node (1s1) [label=below:$$] {};

    \draw (1f1) -- (1s1) -- (1f2);
    \draw (1f3) -- (1s1);
    \end{tikzpicture}}\big\}-line Hoffman graph.
\end{theorem}

Next, we discuss the Johnson graphs. Again using Theorem \ref{intro2}, Koolen et al. showed the following result.

\begin{theorem}[{cf.~\cite[Theorem 1.7]{kyy1}}] \label{Johnson} There exists a positive integer $v^\prime$ such that for each $v \geq v^\prime$, any graph that is cospectral to the Johnson graph $J(v,3)$ is
the slim graph of a $3$-fat \big\{\raisebox{-1ex}{\begin{tikzpicture}[scale=0.3]

\tikzstyle{every node}=[draw,circle,fill=black,minimum size=10pt,scale=0.3,
                            inner sep=0pt]

    \draw (-2.1,0) node (1f1) [label=below:$$] {};
    \draw (-1.6,0) node (1f2) [label=below:$$] {};
    \draw (-1.1,0) node (1f3) [label=below:$$] {};

    \tikzstyle{every node}=[draw,circle,fill=black,minimum size=5pt,scale=0.3,
                            inner sep=0pt]

    \draw (-1.6,1) node (1s1) [label=below:$$] {};

    \draw (1f1) -- (1s1) -- (1f2);
    \draw (1f3) -- (1s1);
    \end{tikzpicture}}\big\}-line Hoffman graph.
\end{theorem}

Van Dam, Haemers, Koolen and Spence \cite[p.~1814]{VanDam.2006} gave a construction of cospectral graphs of the Johnson graph $J(v, D)$ $(v\geq 2D\geq 4)$ that are the block graphs of certain designs.

%Let $D$ be a positive integer and $\mathfrak{h}^{(D)}$ a Hoffman graph with one slim vertex adjacent to $D$ fat vertices. Van Dam, Haemers, Koolen and Spence \cite[p.~1814]{VanDam.2006} gave a construction of cospectral graphs of the Johnson graph $J(v, D)$ $(v\geq 2D\geq 4)$ that are the block graphs of certain designs. %the slim graphs of a $D$-fat $\{\mathfrak{h}^{(D)}\}$-line Hoffman graph.

Similar results can be obtained for cospectral graphs of the Grassmann graphs and the bilinear forms graphs.

\section{Co-edge regular graphs}\label{sec:co-edge regular graphs}
In this section, we discuss co-edge regular graphs with fixed smallest eigenvalue.  A $k$-regular graph of order $n$ is called \emph{co-edge regular} with parameters $(n, k, c)$, if
every pair of distinct and non-adjacent vertices has exactly $c$ common neighbors. Note that a $(t_1\times t_2)$-grid, which is the line graph of a complete bipartite graph $K_{t_1,t_2}$, is a co-edge regular graph with parameters $(t_1t_2,t_1+t_2-2,2)$.

The results in this section are motivated by two results of Terwilliger for distance-regular graphs.
The first result concerns the local graph of a thin $Q$-polynomial distance-regular graph. For the definition of a thin $Q$-polynomial distance-regular graph, see Appendix \ref{appendix}.
\begin{lemma}[cf. {\cite[Theorem 77]{Terwilliger.1993}}]\label{local}
Let $G$ be a thin $Q$-polynomial distance-regular graph with diameter $D\geq5$. Then there exists a non-negative integer $c$ such that for each vertex $x$ of $G$, the local graph $\Delta_G(x)$ at $x$ is co-edge regular with parameters $(k, a_1, c)$. Moreover, the local graph $\Delta_G(x)$ at $x$ has at most $5$ distinct eigenvalues.
\end{lemma}

The second result shows a relation between the second largest eigenvalue of a distance-regular graph and the smallest eigenvalue of its local graph at any vertex.
\begin{lemma}[{cf.~\cite[Theorem 4.4.3]{bcn}}]\label{laminloc}
Let $G$ be a distance-regular graph with diameter $D\ge 3$ and second largest eigenvalue $\lambda_1$. Then for each vertex $x$ of $G$, the local graph $\Delta_G(x)$ at $x$ has the smallest eigenvalue at least $-\frac{b_1}{\lambda_1+1}-1$.
\end{lemma}

Now we present an important tool shown by Yang and Koolen \cite{Yang.2021},  by using Hoffman graphs.
\begin{proposition}[{\cite[Proposition 1.3]{Yang.2021}}]\label{tool}

Let $\lambda\geq2$ be a real number. Then there exists a constant $M_1(\lambda) \geq \lambda^3$ such that, if a graph $G$ satifies
\begin{enumerate}
\item every pair of vertices at distance $2$ has at least $M_1(\lambda)$ common neighbors, and
\item $\lambda_{\min}(G)\geq-\lambda$,
\end{enumerate}
then $G$ has diameter $2$ and for each $x$, the number of vertices at distance $2$ to $x$ is at most $\lfloor \lambda \rfloor \lfloor(\lambda-1)^2 \rfloor$.
\end{proposition}

Next we give an upper bound on the parameter $c$ for co-edge regular graphs in terms of its smallest eigenvalue, due to Yang and Koolen \cite{Yang.2021}.
They showed the following:
\begin{theorem}[{\cite[Theorem 7.1]{Yang.2021}}]\label{YK1}
Let $\lambda \geq 2$ be a real number. There exists a real number $M_2(\lambda)$ such that, for any connected co-edge regular graph $G$ with parameters $(n, k, c)$, if $\lambda_{\min}(G)\geq-\lambda $, then $c > M_2(\lambda)$ implies that  $n-k-1 \leq \frac{(\lambda -1)^2}{4}+1$ holds.
\end{theorem}

Their proof used Proposition \ref{tool} and the Alon-Boppana Theorem.
Later Koolen, Gebremichel and Yang \cite{Koolen.2019} observed that this result is also true when the condition co-edge regular is replaced by sesqui-regular, which we are going to introduce.
A $k$-regular graph of order $n$ is called \emph{sesqui-regular} with parameters $(n, k, c)$, if every pair of vertices at distance $2$ has exactly $c$ common neighbors.

They further extended this result as follows:
\begin{theorem}\label{main sesqui}
Let $\lambda \geq 2$ be an integer. There exists a constant $C_4(\lambda)$ such that, for any sesqui-regular graph $G$ with parameters $(n, k, c)$, if $\lambda_{\min}(G) \geq -\lambda$ and $k \geq C_4(\lambda)$, then one of the following holds:
\begin{itemize}
\item[\rm(i)] $c \leq \lambda^2(\lambda-1)$,
\item[\rm(ii)] $n-k-1 \leq \frac{(\lambda -1)^2}{4}+1$.
\end{itemize}
\end{theorem}

\begin{remark}
\begin{enumerate}[(i)]
%\item Note that with Proposition \ref{tool}, one can show a weaker version of the Alon-Boppana Theorem for triangle-free regular graphs.
\item The block graph of a Steiner system $S(2,t,v)$ is strongly regular with parameters $(\frac{v(v-1)}{t(t-1)},\frac{t(v-t)}{t-1},(t-1)^2+\frac{v-1}{t-1}-2, t^2)$ and smallest eigenvalue $-t$. This shows that the bound in the first item of the above theorem can not be improved too much.
\item There exists an infinite family of bipartite Ramanujan graphs with valency $k$ and unbounded number of vertices, which were found by Marcus et al. \cite{Marcus.2015}, as we already discussed in Section \ref{sec:alon-boppana}.  Let $G$ be a graph in this family, say of order $n$. Consider the complement $\overline{G}$ of $G$.
Then the following holds:
\begin{enumerate}[(a)]
\item $\overline{G}$ has valency $n-k-1$;
\item $\overline{G}$ has smallest eigenvalue at least $-1-2\sqrt{k-1}$ as $G$ has second largest eigenvalue at most $2\sqrt{k-1}$.
\end{enumerate}
This shows that the bound in the second item of Theorem \ref{main sesqui} is tight.
\item  In \cite{Koolen.2021}, it is shown that for $\lambda =3$, one can improve the bound in the first item of the above theorem to $\lambda^2$. Whether this is true for general $\lambda$, it is not known.
%\item In Yang and Koolen \cite{Koolen.2019}, they give some results on the  number $v(k, \lambda)$ when $\lambda$ is fixed and $k$ goes to infinity.
\end{enumerate}
\end{remark}

Now we give some spectral characterizations of some families of graphs under the extra assumption that the graphs are co-edge regular. Hayat, Koolen and Riaz \cite{Hayat.2019} showed the following result for the clique-extensions of the square grid graphs.
\begin{theorem}[{\cite[Theorem 1.1]{Hayat.2019}}]\label{main-result}
Let $G$ be a co-edge regular graph with spectrum
\begin{equation*}
\left\{\big(s(2t+1)-1\big)^{1},(st-1)^{2t},(-1)^{(s-1)(t+1)^{2}},(-s-1)^{t^2}\right\},
\end{equation*} where $s \geq 2, t\geq 1$ are integers.
If $t\geq11(s+1)^3(s+2)$, then $G$ is the $s$-clique extension
of the $((t+1)\times(t+1))$-grid.
\end{theorem}
Tan, Koolen and Xia \cite{Tan.2020} showed a similar result for the clique-extensions of triangular graphs:
\begin{theorem}[{\cite[Theorem 1]{Tan.2020}}]\label{thm1}
Let $G$ be a co-edge regular graph with spectrum
$$\left\{(2sv-3s-1)^1,(sv-3s-1)^{v-1},(-s-1)^{\frac{v^2-3v}{2}},(-1)^{\frac{(s-1)v(v-1)}{2}}\right\},$$
where $s\geq2$ and $v\geq1$ are integers. If $v\geq 48s$, then $G$ is the $s$-clique extension of the triangular graph $T(v)$.
\end{theorem}
\begin{remark}
\begin{enumerate}
\item Note that to prove both Theorems \ref{main-result} and \ref{thm1}, they used the claw-clique method of Bose and Laskar.
%\item  The method used to show Theorem \ref{thm1} is an improvement over the method
%used to show Theorem \ref{main-result}.
\item Using the method in \cite{Tan.2020}, one could improve the bound $t\geq11(s+1)^3(s+2)$ of Theorem \ref{main-result}.
%\item Gavrilyuk and Koolen \cite{Gavrilyuk.2018} used a weaker version of Theorem \ref{main-result} to show that the Grassmann graph $J_q(2D,D)$ is determined by its intersection array, if $D$ is much larger then $q$.
\end{enumerate}
\end{remark}

Tan et al. \cite{Tan.2020} also gave the following conjecture:
\begin{conjecture}[{\cite[Conjecture 3]{Tan.2020}}]
Let $G$ be a connected co-edge regular graph with parameters $(n, k,c)$ having four distinct eigenvalues. Let $\lambda\ge 2$ be an integer. Then there exists a constant $n_2(\lambda)$ such that,
if $\lambda_{\min}(G)\ge -\lambda$, $n\geq n_2(\lambda)$ and $k< n-2-\frac{(\lambda-1)^2}{4}$, then either $G$ is the
$s$-clique extension of a strongly regular graph for $2\leq s\leq \lambda-1$ or $G$
is a $(p\times q)$-grid with $p>q\geq 2$.
\end{conjecture}

Yang, Abiad and Koolen \cite{Yang.2017} showed the following spectral characterization of $2$-clique extensions of the square grid graphs, using Hoffman graphs. They did not need
the assumption of co-edge regularity, but they needed a very large lower bound on the valency.
\begin{theorem}[{\cite[Theorem 1]{Yang.2017}}]
The $2$-clique extension of the $(t\times t)$-grid is characterized by its spectrum if $t$ is large enough.
\end{theorem}
In their proof, they used the following result of Koolen et al. \cite{kyy1}.
\begin{theorem}[{\cite[Theorem 1.8]{kyy1}}]\label{grid} There exists a positive integer $t$ such that for each pair $(t_1,t_2)$ with $t_1 \geq t_2 \geq t$, any graph that is cospectral to the $2$-clique extension of the $(t_1 \times t_2)$-grid is the slim graph of a $2$-fat \big\{\raisebox{-1ex}{\begin{tikzpicture}[scale=0.3]

\tikzstyle{every node}=[draw,circle,fill=black,minimum size=10pt,scale=0.3,
                            inner sep=0pt]

    \draw (-2.1,0) node (1f1) [label=below:$$] {};
    \draw (-1.6,0) node (1f2) [label=below:$$] {};
    \draw (-1.1,0) node (1f3) [label=below:$$] {};

    \tikzstyle{every node}=[draw,circle,fill=black,minimum size=5pt,scale=0.3,
                            inner sep=0pt]

    \draw (-1.6,1) node (1s1) [label=below:$$] {};

    \draw (1f1) -- (1s1) -- (1f2);
    \draw (1f3) -- (1s1);
    \end{tikzpicture}},\hspace{-0.08cm}
\raisebox{-1ex}{\begin{tikzpicture}[scale=0.3]
\tikzstyle{every node}=[draw,circle,fill=black,minimum size=10pt,scale=0.3,
                            inner sep=0pt]

    \draw (-0.5,0) node (2f1) [label=below:$$] {};
    \draw (0.5,0) node (2f2) [label=below:$$] {};
    \draw (-0.5,1) node (2f3) [label=below:$$] {};
    \draw (0.5,1) node (2f4) [label=below:$$] {};

    \tikzstyle{every node}=[draw,circle,fill=black,minimum size=5pt,scale=0.3,
                            inner sep=0pt]

    \draw (0,0.2) node (2s1) [label=below:$$] {};
    \draw (0.3,0.5) node (2s2) [label=below:$$] {};
    \draw (-0.3,0.5) node (2s3) [label=below:$$] {};
    \draw (0,0.8) node (2s4) [label=below:$$] {};

    \draw (2f1) -- (2s1) -- (2f2) -- (2s2) -- (2f4) -- (2s4) -- (2f3) -- (2s3) -- (2f1);
    \draw (2s1) -- (2s4);
    \draw (2s2) -- (2s3);
    \end{tikzpicture}},\hspace{-0.08cm}
\raisebox{-1ex}{\begin{tikzpicture}[scale=0.3]
\tikzstyle{every node}=[draw,circle,fill=black,minimum size=10pt,scale=0.3,
                            inner sep=0pt]

    \draw (1.5,0) node (3f1) [label=below:$$] {};
    \draw (1.5,1) node (3f2) [label=below:$$] {};

    \tikzstyle{every node}=[draw,circle,fill=black,minimum size=5pt,scale=0.3,
                            inner sep=0pt]

    \draw (1,0.5) node (3s1) [label=below:$$] {};
    \draw (2,0.5) node (3s2) [label=below:$$] {};

    \draw (3s1) -- (3s2);
    \draw (3f1) -- (3s1) -- (3f2) -- (3s2) -- (3f1);
    \end{tikzpicture}}\big\}-line Hoffman graph.
\end{theorem}

\section{Signed graphs}\label{sec:signed graph}
\subsection{Definitions}
\begin{definition}[signed graph]
A signed graph $(G, \tau)$ is a pair of a graph $G =(V(G), E(G))$ and a signing $\tau: E(G) \rightarrow \{+1,-1\}$.
\end{definition}

The \emph{adjacency matrix} $A(G,\tau)$ of the signed graph $(G,\tau)$ is the symmetric matrix whose rows and columns are indexed by $V(G)$ such that $A(G,\tau)_{x,y} = \tau(\{x,y\})$ if $\{x,y\}$ is an edge of $G$ and $0$ otherwise.

A real number $\lambda$ is an \emph{eigenvalue} of $(G, \tau)$, if $\lambda$ is an eigenvalue of its adjacency matrix $A(G, \tau)$. The \emph{spectrum} of $(G, \tau)$ is the spectrum of $A(G, \tau)$. In a similar fashion, we denote by $\lambda_{\min}(G, \tau)$ the smallest eigenvalue of the signed graph $(G, \tau)$.

For $\varepsilon \in \{+,-\}$, the \emph{$\varepsilon$-graph} of $(G,\tau)$ is the graph $(G, \tau)^{\varepsilon}$ with vertex set $V(G)$ and edge set $E^{\varepsilon}$, where $E^{\varepsilon}= \{ e\in E(G) \mid \tau(e) =\varepsilon1\}$. We also represent the signed graph $(G,\tau)$ by the triple $(V(G),E^{+},E^{-})$.

Two signed graphs $(G, \tau)$ and $(H, \xi)$ are \emph{switching equivalent} if there exist a permutation matrix $P$ and a diagonal matrix $D$ with diagonal entries in $\{-1, +1\}$ such that $P A(G,\tau) P^T = D A(H,\xi)D$ holds.

\subsection{Seidel matrices}
In this subsection, we introduce and study Seidel matrices.
\begin{definition}[Seidel matrix]
\begin{enumerate}
\item A Seidel matrix $S$ of order $n$ is a symmetric $(0, \pm1)$-matrix with $0$ on the diagonal and $\pm1$ otherwise.
\item Let $G$ be a graph. The Seidel matrix $S(G)$ of $G$ is the matrix $J-I -2A(G)$, where $A(G)$ is the adjacency matrix of $G$.
\end{enumerate}
\end{definition}
Note that the Seidel matrix $S(G)$ of a graph $G$ of order $n$ is the adjacency matrix of the signed graph $(K_n, \tau_G)$ satisfying $(K_n, \tau_G)^-= G$. Therefore, for a given graph $G$ of order $n$, we denote by $(K_n, \tau_G)$ the signed graph with $S(G)$ as its adjacency matrix. Two graphs $G$ and $H$ of order $n$ are called \emph{switching equivalent} if the signed graphs $(K_n, \tau_G)$ and $(K_n, \tau_H)$ are switching equivalent. If the graphs $G$ and $H$ are switching equivalent, then there exists a subset $V'$ of the vertex set $V(G)$ of $G$, such that the resulting graph, by changing all the edges between $V'$ and $V(G)-V'$ to non-edges, and all the non-edges between $V'$ and $V(G)-V'$ to edges, is $H$. This operation is called \emph{switching} on the subset $V'$.

Note that switching equivalence is an equivalence relation. The equivalence class of $G$ is called the \emph{switching class} of $G$ and is denoted by $[G]$.
If $H \in [G]$, the matrices $S(H)$ and $S(G)$ are similar and hence have the same spectrum.

The motivation to study Seidel matrix with fixed smallest eigenvalue with large multiplicity comes from the study of equiangular lines in the Euclidean space.
\begin{definition}[equiangular lines]
A system of lines through the origin in the $r$-dimensional Euclidean space $\mathbb{R}^r$ is called equiangular if the angle between any pair of lines is the same.
\end{definition}

Seidel matrices and systems of equiangular lines, are related as follows (see for example, \cite[Section 11.1]{GD01}):
\begin{proposition}\label{eq}
Let $n>r\geq2$ be integers. There exists a system of $n$ equiangular lines in $\mathbb{R}^r$ with common angle $\arccos \alpha$ if and only if there exists a Seidel matrix $S$ of order $n$ such that $S$ has smallest eigenvalue at least $-\frac{1}{\alpha}$ and \rm{rank}$(S+\frac{1}{\alpha}I)\leq r$.
\end{proposition}

We are going to use the theory of Hoffman graphs to study Seidel matrices with fixed smallest eigenvalue. Our approach is different from, but closely related to the approach of Balla,
Dr{\"a}xler, Keevash and Sudakov \cite[Section 2]{Balla.2018}.

Let $S$ be a Seidel matrix of order $n$. The graph $G^+(S)$ is the graph with adjacency matrix $\frac{1}{2}(S +J-I)$. If $\lambda_{\min}(S) = 2\lambda +1$, then $\lambda_{\min}(G^+(S)) \geq \lambda$. Note that for each graph in $[G^+(S)]$, its smallest eigenvalue is at least $\lambda$. Let $C$ be a clique of order $q$ in $H\in[G^+(S)]$. If necessary, by switching, we can obtain a graph in $[G^+(S)]$ such that $C$ is still a clique and every vertex outside $C$ has at least $q/2$ neighbors in $C$. We denote such a graph by $H_{C}$.

\begin{theorem} \label{seidelthm}
Let $\lambda \leq -2$ be a real number and $r$ a positive integer.  Let $m$ be such that the smallest eigenvalue of $\widetilde{K}_{2m}$ is less than $\lambda$.
Then there exists a positive integer $Q=Q(\lambda, m, r) \geq (m+1)^2$ such that for each integer $q \geq Q$ and each Seidel matrix $S$ with $\lambda_{\min}(S)\geq2\lambda+1$, if a graph $H\in [G^+(S)]$ contains a clique $C$ of order at least $q$, then the associated Hoffman graph $\mathfrak{g} = \mathfrak{g}(H_C, m, q)$ is $(r, \lambda)$-nice and has exactly one fat vertex and this fat vertex is adjacent to all slim vertices.

Moreover there exists a positive integer $n= n(\lambda,q)$ such that if the order of the Seidel matrix $S$ is at least $n$, then every graph in $[G^+(S)]$ contains a clique of order at least $q$.
\end{theorem}

\begin{proof}
Let $Q:=Q(\lambda,m,r)$ be the integer such that Theorem \ref{nicethm} holds. Let $q\geq Q$. Suppose a graph $H\in [G^+(S)]$ has a clique $C$ of order at least $q$. Without loss of generality, we may assume $C$ is a maximal clique in $H$.  Now we look at a graph $H_C$. It is known that all vertices outside $C$ have at least $q/2$ neighbors in $C$. We claim that every vertex outside $C$ has at most $m-1$ non-neighbors in $C$. Suppose this is not the case, and the  vertex $x$ outside $C$ has at least $m$ non-neighbors in $C$. Then we can find $m$ vertices $y_1,y_2,\ldots,y_m$ in $C$ which are not adjacent to $x$. Since $x$ also has at least $q/2$ neighbors in $C$, where $q/2\geq (m+1)^2/2\geq m$, we can find $z_1,z_2,\ldots,z_m$ in $C$ which are adjacent to $x$. It is not hard to see that the subgraph of $H_C$ induced on $\{x,y_1,y_2,\ldots,y_m,z_1,z_2,\ldots,z_m\}$ is $\widetilde{K}_{2m}$, which has smallest eigenvalue less than $\lambda$. This is not possible, as $\lambda_{\min}(H_C)\geq\lambda$. Hence, the claim holds, and under the equivalence relation $\equiv_q^m$, the set of maximal cliques in $H_C$ of order at least $q$ has exactly one equivalence class, which consists of all vertices of $H_C$. This means that the associated Hoffman graph $\mathfrak{g}(H_C, m, q)$ has exactly one fat vertex and this fat vertex is adjacent to all the slim vertices of $\mathfrak{g}(H_C, m, q)$. Considering $q\geq Q(\lambda,m,r)$, we have that $\mathfrak{g}(H_C, m, q)$ is $(r, \lambda)$-nice by Theorem \ref{nicethm}.

Let $n:=R(q,\lceil-2\lambda\rceil+1)$, where $R(q, \lceil-2\lambda\rceil+1)$ is the Ramsey number. For a graph $H'\in[G^+(S)]$ with Seidel matrix $S(H')$, we claim that $H'$ does not contain a stable set of size $\lceil-2\lambda\rceil+1$. Otherwise, the matrix $S(H')$ contains a $((\lceil-2\lambda\rceil+1)\times (\lceil-2\lambda\rceil+1))$ principal submatrix $J_{\lceil-2\lambda\rceil+1}-I_{\lceil-2\lambda\rceil+1}$. In other words, the matrix $-S(H')$ contains a principal submatrix $-J_{\lceil-2\lambda\rceil+1}+I_{\lceil-2\lambda\rceil+1}$. Note that the matrix $-S(H')$ has the same spectrum as $S$. Thus \[\lceil2\lambda\rceil=\lambda_{\min}(-J_{\lceil-2\lambda\rceil+1}+I_{\lceil-2\lambda\rceil+1})\geq\lambda_{\min}(-S(H'))\geq2\lambda+1,\]  which gives a contradiction. (For the first inequality, see \cite[Theorem 9.1.1]{Godsil.2016}.) Therefore, our claim holds and by Ramsey theory, the graph $H'$ contains a clique of order $q$.

This completes the proof.
\end{proof}

\subsection{Signed graphs}
Following a straightforward way, the notion of the $s$-integrability of graphs can be extended to signed graph. For a positive integer $s$, we say that a signed graph $(G,\tau)$ with smallest eigenvalue $\lambda_{\min}(G,\tau)$ is \emph{$s$-integrable}, if there exists an integral matrix $N$ such that the equality
\[
s(A(G,\tau)+\lceil-\lambda_{\min}(G,\tau)\rceil I)=N^TN
\]
holds, where $A(G,\tau)$ is the adjacency matrix of $(G,\tau)$.

Not so much is known about signed graphs with fixed smallest eigenvalue. Using the same proof as in Theorem \ref{thmcam} one can show:
\begin{theorem}[{\cite[Theorem 3.13]{Belardo.2018}}]\label{thmcam2}
If $(G, \tau)$ is a connected signed graph with $\lambda_{\min}(G, \tau) \geq -2$, then $(G,\tau)$ is $s$-integrable for $s \geq 2$. Moreover, if $(G,\tau)$ has at least $121$ vertices, then $(G,\tau)$ is $1$-integrable.
\end{theorem}

Theorem \ref{thmhoff77} \rm{(i)} has been extended to the class of signed graphs by Gavrilyuk, Munemasa, Sano and Taniguchi \cite{Gavrilyuk.2020} as follows.
\begin{theorem}
 For any real number $\lambda\in (-2,-1]$, there exists a constant $C_5(\lambda)$ such that if $(G,\tau)$ is a connected signed graph on $n$ vertices with
 $\lambda_{\min}(G, \tau)\geq\lambda$ and minimal valency at least $C_5(\lambda)$,
  then $\lambda_{\min}(G,\tau) = -1$ and $(G,\tau)$ is switching equivalent to $(K_n, +)$, and hence is $1$-integrable.
\end{theorem}

\section{Future work}\label{sec:future work}
At the end of this survey, we give some problems for discussion.

\subsection{Problems on graphs and signed graphs}

First we discuss unsigned graphs. As a generalization of Theorem \ref{mthmkyy3}, we have
\begin{problem}[{\cite[Conjecture 3.2]{Koolen.2020}}]\label{pr1}
There exist constants $\kappa_5$ and $s_1$ such that, for any graph $G$ with $\lambda_{\min}(G)\geq -4$ and minimal valency at least $\kappa_5$, $G$ is $s_1$-integrable.
\end{problem}

To solve this problem, we have to have a good understanding of fat Hoffman graphs with smallest eigenvalue at least $-4$. The first step is to look at the family $\mathfrak{H}$ of indecomposable fat Hoffman graphs with smallest eigenvalue at least $-4$, in which every slim vertex has exactly one fat neighbor and any two distinct slim vertices have no common fat neighbor. Note that for any Hoffman graph in $\mathfrak{H}$, its slim graph is a connected graph with smallest eigenvalue at least $-3$. Therefore, as a subproblem of Problem \ref{pr1}, we have

\begin{problem}[{\cite[Conjecture 3.1]{Koolen.2020}}]
There exists a constant $s_2$ such that, for any graph $G$ with $\lambda_{\min}(G)\geq -3$, $G$ is $s_2$-integrable.
\end{problem}

In \cite{KRY2}, Koolen et al. showed that the complement of the McLaughlin graph has smallest eigenvalue $-3$ and can not be $2$-integrated, but it is $4$-integrable. It is not clear whether any graph with smallest eigenvalue at least $-3$ is always $4$-integrable. The situation for graphs with smallest eigenvalue at least $-3$ is different from the situation for graphs with smallest eigenvalue at least $-2$, as there exists an infinite family of connected graphs with unbounded vertices such that none of them can be $2$-integrated (see \cite[Remark 1.4 \rm{(v)}]{kyy3}).

As a special case, we look at the integrability of trees with smallest eigenvalue at least $-3$.
\begin{problem}
Is it true that any tree with smallest eigenvalue at least $-3$ is $2$-integrable?
\end{problem}
In \cite{Koolen.2017}, Koolen, Rehman and Yang characterized $1$-integrable trees with smallest eigenvalue at least $-3$.

For signed graphs we propose the following problems.
\begin{problem}
 For any real number $\lambda\in (-1-\sqrt{2},-2]$, there exists a constant $C_6(\lambda)$ such that, if $(G,\tau)$ is a signed graph with
 $\lambda_{\min}(G, \tau)\geq\lambda$ and minimal valency at least $C_6(\lambda)$,
  then $\lambda_{\min}(G,\tau) =-2$ and $(G,\tau)$ is $1$-integrable.
\end{problem}
\begin{problem}
There exists a constant $\kappa_6$ such that, if $(G,\tau)$ is a signed graph with $\lambda_{\min}(G, \tau)\geq-3$ and minimal valency at least $\kappa_6$, then $(G,\tau)$ is $s$-integrable for any $s\geq2$.
\end{problem}

To solve above two problems, one problem to overcome is how to deal with switching classes of signed graphs. Is there a natural generalization of Hoffman graphs to the class of signed graphs? In \cite{Gavrilyuk.2020}, they considered a generalization of line Hoffman graphs.
For more problems on graphs and signed graphs with fixed smallest eigenvalue, see \cite{Koolen.2020}. For more general problems on the spectral theory of signed graphs, see \cite{Belardo.2018}.

\subsection{Refining \texorpdfstring{$\widetilde{K}_{2t}$}{\widetilde{K}_{2t}}}

Denote by $H(a,t)$ the graph on $a+t+1$ vertices consisting of a clique $K_{a+t}$ together with a vertex that is adjacent to precisely $a$ vertices of this clique. Note that $\widetilde{K}_{2t}$ is exactly the graph $H(t,t)$.

In \cite{Greaves.2020}, Greaves, Koolen, and Park showed the following.

\begin{lemma}\label{lem:hat}
  Let $G$ be a graph having smallest eigenvalue $-m$ that contains $H(a,t)$ as an induced subgraph.
  Then
  \begin{equation}
    \label{eqn:atbound}
    (a-m(m-1))(t-(m-1)^2) \leqslant (m(m-1))^2.
  \end{equation}
\end{lemma}
Using the above result, they obtain bounds on the clique order in strongly regular graphs. Their result can also be extended to other classes of graphs, for example, distance-regular graphs.

Lemmens and Seidel \cite{Lemmens.1973} conjectured that for each Seidel matrix $S$ of order $n$, the rank of the matrix $S +5I$ is at least $\lfloor\frac{2n}{3}\rfloor +1$. This was shown to be true by Cao, Koolen, Lin and Yu \cite{Cao.2020}. Their main tool is to use the Seidel matrices of the complements of $H(a, t)$'s as minimal forbidden principal submatrices.

\section*{Acknowledgments}
J.H. Koolen is partially supported by the National Natural Science Foundation of China (No. 12071454) and Anhui Initiative in Quantum Information Technologies (No. AHY150000).

Q. Yang is partially supported by the Fellowship of China Postdoctoral Science Foundation (No. 2020M671855).

We greatly thank Prof.~Min Xu for supporting M.-Y. Cao to visit University of Science and Technology of China.

We are also grateful to Prof.~Sebastian M. Cioabă, Prof.~Akihiro Munemasa, Dr.~Jongyook Park and Mr.~Kiyoto Yoshino for their careful reading and valuable comments.

\titleformat{\section}{\large\bfseries}{\appendixname~\thesection .}{0.5em}{}
\appendix
\section{$Q$-polynomial distance-regular graphs}\label{appendix}
Let $V$ denote a non-empty finite set.
Let $\mathrm{Mat}_V(\mathbb{C})$ denote the $\mathbb{C}$-algebra consisting of all complex matrices whose rows and columns are indexed by $V$.
Let $\mathds{U}=\mathbb{C}^V$ denote the $\mathbb{C}$-vector space consisting of all complex vectors indexed by $V$.
We endow $\mathds{U}$ with standard Hermitian inner product $(\mathbf{u},\mathbf{v})=\mathbf{u}^T\overline{\mathbf{v}}$ for $\mathbf{u},\mathbf{v} \in \mathds{U}$.
We view $\mathds{U}$ as a left module for $\mathrm{Mat}_V(\mathbb{C})$, called the \emph{standard module}.

Let $G$ be a distance-regular graph of diameter $D$. Let $V$ be the vertex set of $G$. For $0\leq i\leq D$, let $A_i$ denote the matrix in $\mathrm{Max}_{V}(\mathbb{C})$ defined by
\[
(A_i)_{x,y}=\left\{
\begin{array}{ll}
  1 & \text{ if }d(x,y)=i, \\
  0 & \text{ otherwise},
\end{array}
\right.
\]
where $x,y\in V$. We call $A_i$ the \emph{ $i$th distance matrix} of $G$. We abbreviate $A:=A_1$. Observe that
\begin{itemize}
\item[(1a)] $A_0=I$;
\item[(1b)] $\sum^D_{i=0}A_i=J$, the all-ones matrix;
\item[(1c)] each $A_i$ is real symmetric;
\item[(1d)] there exist $p_{ij}^h$ for $0\leq i,j,h\leq D$, such that $A_iA_j =A_jA_i= \sum^D_{h=0}p^h_{ij}A_h$ hold.
\end{itemize}
Notice that (\rm{1a}) implies for each pair vertices $x,y\in V$ with $d(x,y)=h$, the equality $|G_i(x)\cap G_j(y)|=p_{ij}^h$ holds. Therefore, for all integers $0\leq h,i,j\leq D$, $p^h_{ij}=0$ (resp.~$p^h_{ij}\neq0$) if one of $h,i,j$ is greater than (resp.~equal to) the sum of the other two. By these facts, we find that $A_0, A_1,\ldots, A_D$ is a basis for a commutative subalgebra $M$ of $\mathrm{Mat}_V(\mathbb{C})$, which we call the \emph{Bose-Mesner algebra} of $G$.
It is known that $A$ generates $M$, as $AA_i=c_{i+1}A_{i+1}+a_iA_i+b_{i-1}A_{i-1}$ ($0\leq i\leq D$) by condition (\rm{iv}), where $\{b_0,b_1,\ldots,b_{D-1};c_1,c_2,\ldots,c_D\}$ is the intersection array of $G$.

The algebra $M$ has a second basis $E_0, E_1, \ldots, E_D$ such that
\begin{itemize}
\item[(2a)] $E_0=|V|^{-1}J$,
\item[(2b)] $\sum^D_{i=0}E_i=I$,
\item[(2c)] each $E_i$ is real symmetric,
\item[(2d)] $E_iE_j=E_jE_i=\delta_{ij}E_i$
\end{itemize}
(see \cite[p.~45]{bcn}).
We call $E_i$ the \emph{$i$th primitive idempotent} of $G$.
Since $\{E_i\}^D_{i=0}$ is a basis for $M$, there exist complex scalars $\{\theta_i\}^D_{i=0}$ such that $A=\sum^D_{i=0}\theta_i E_i$. (Note that $\{\theta_i\}^D_{i=0}$ are exactly all of the distinct eigenvalues of $G$ and they are real.) Observe $AE_i=E_iA=\theta_iE_i$ for $0\leq i \leq D$. We call $\theta_i$ the eigenvalue of $G$ associated with $E_i$ for $0\leq i \leq D$.
Observe $\mathds{U}=E_0\mathds{U}\oplus E_1\mathds{U}\oplus\cdots\oplus E_D\mathds{U}$, an orthogonal direct sum.
For $0\leq i \leq D$, $E_i\mathds{U}$ is the eigenspace of $A$ associated with $\theta_i$.
Denote by $m_i$ the rank of $E_i$ and observe $m_i=\dim(E_i\mathds{U})$, the multiplicity of the eigenvalue $\theta_i$.

We now introduce the notion of $Q$-polynomial property of $G$.
Let $\circ$ denote the entrywise product in $\mathrm{Mat}_V(\mathbb{C})$. Since $A_i\circ A_j =\delta_{ij}A_i$, the Bose-Mesner algebra $M$ is closed under $\circ$.
Also as $\{E_i\}^D_{i=0}$ is a basis for $M$, there exist complex scalars $q^h_{ij}$ such that \[E_i\circ E_j = |V|^{-1}\sum^D_{h=0}q^h_{ij}E_h.\]
By \cite[p.~48, p.~49]{bcn}, the scalars $q^h_{ij}$ are real and non-negative.
We say $G$ is \emph{$Q$-polynomial} (with respect to the given ordering $E_0, E_1, \ldots, E_D$) whenever for all integers $0\leq h,i,j\leq D$, $q^h_{ij}=0$ (resp.~$q^h_{ij}\ne0$) if one of $h,i,j$ is greater than (resp.~equal to) the sum of the other two \cite[p.~235]{bcn}.

We assume $G$ is $Q$-polynomial with respect to the ordering $E_0, E_1, \ldots, E_D$.
Fix a vertex $x \in V$. We refer to $x$ as a ``base'' vertex.
For $0 \leq i \leq D$, we define the diagonal matrix $E^*_i=E^*_i(x) \in \mathrm{Mat}_V(\mathbb{C})$ with diagonal entry
\[
	(E^*_i)_{y,y} =\left\{
   \begin{array}{ll}
   1 & \text{ if }d(x,y)=i, \\
  0 & \text{ otherwise},
   \end{array}
   \right.
\]
where $y\in V$. We call $E^*_i$ the \emph{$i$th dual primitive idempotent} of $G$ with respect to $x$.
Observe
\begin{itemize}
\item[(3a)] $\sum^D_{i=0}E^*_i=I$,
\item[(3b)] each $E^*_i$ is real symmetric,
\item[(3c)] $E^*_iE^*_j=\delta_{ij}E^*_i$.
\end{itemize}
By these facts, $E^*_0, E^*_1, \ldots, E^*_D$ is a basis for a commutative subalgebra $M^*$ of $\mathrm{Mat}_V(\mathbb{C})$, which we call the \emph{dual Bose-Mesner algebra} of $G$.

Define the diagonal matrix $A^*_i=A^*_i(x) \in\mathrm{Mat}_V(\mathbb{C})$ with diagonal entry $(A^*_i)_{y,y} = |V|(E_i)_{x,y}$ for $y\in V$.
By \cite[p.~379]{Terwilliger.1992}, $A^*_0, A^*_1, \ldots, A^*_D$ is also a basis for $M^*$, and moreover
\begin{itemize}
\item[(4a)] $A^*_0=I$,
\item[(4b)] $\sum^D_{i=0}A^*_i=|V|E_0^*$,
\item[(4c)] each $A_i^*$ is real and symmetric,
\item[(4d)] $A^*_iA^*_j=A^*_jA^*_i=\sum^D_{h=0} q^h_{ij} A^*_h$.
\end{itemize}
We call $A^*_i$ the $i$th \emph{dual distance matrix} of $G$ with respect to $x$. We abbreviate $A^*=A^*_1$, called the \emph{dual adjacency matrix} of $G$ with respect to $x$. From conditions (\rm{4a}) and (\rm{4d}), we find that the matrix $A^*$ generates $M^*$.
Since $\{E^*_i\}^D_{i=0}$ is a basis for $M^*$, there exist complex scalars $\{\theta^*_i\}^D_{i=0}$ such that $A^*=\sum^D_{i=0}\theta^*_iE^*_i$.
Observe $A^*E^*_i=E^*_iA^*=\theta^*_iE^*_i$ for $0 \leq i \leq D$.
The scalars $\{\theta^*_i\}^D_{i=0}$ are real \cite[Lemma 3.11]{Terwilliger.1992} and mutually distinct.
We call $\theta^*_i$ the \emph{dual eigenvalue} of $G$ associated with $E^*_i$.
Observe $\mathds{U}=E^*_0\mathds{U}\oplus E^*_1\mathds{U}\oplus \cdots \oplus E^*_D\mathds{U}$, an orthogonal direct sum.
For $0 \leq i \leq D$, the space $E^*_i\mathds{U}$ is the eigenspace of $A^*$ associated with $\theta^*_i$.

Let $T=T(x)$ denote the subalgebra of $\mathrm{Mat}_V(\mathbb{C})$ generated by $M$ and $M^*$.
We call $T$ the \emph{Terwilliger algebra} (or \emph{subconstituent algebra}) of $G$ with respect to $x$ \cite{Terwilliger.1992}. Note that $A$ and $A^*(x)$ generates $T$.
The algebra $T$ is finite dimensional and non-commutative. It is also semi-simple since it is closed under conjugate and transpose map.
The following are relations in $T$ \cite[Lemma 3.2]{Terwilliger.1992}.
For $0 \leq h,i,j \leq D$,
\begin{align*}
	E^*_iA_hE^*_j = 0 \quad \text{if and only if} \quad p^h_{ij}=0,\\
	E_iA^*_hE_j = 0 \quad \text{if and only if} \quad q^h_{ij}=0.
\end{align*}
Note that $T$ may depend on the choice of the base vertex (see \cite{Bang.2009}).

By a \emph{$T$-module}, we mean a subspace $\mathds{W}$ of $\mathds{U}$ such that $B\mathds{W}\subseteq \mathds{W}$ for all $B\in T$.
Observe that $\mathds{U}$ is a $T$-module, called the \emph{standard module of $T$} (or \emph{standard $T$-module}). A $T$-module is called \emph{irreducible} if it contains no $T$-submodule except itself and zero module.

Let $\mathds{W}$ be a $T$-module and $\mathds{W}_1$ a $T$-submodule of $\mathds{W}$.
Then the orthogonal complement of $W_1$ in $W$ is a $T$-module, since $T$ is closed under conjugate transpose map.
It follows that $W$ decomposes into an orthogonal direct sum of irreducible $T$-modules.

Let $W$ denote an irreducible $T$-module.
Then $W$ decomposes into a direct sum of nonzero spaces among $E^*_iW$, $0 \leq i \leq D$.
By the \emph{endpoint} of $W$, we mean $\min\{i \mid 0 \leq i \leq D, E^*_iW\ne 0\}$.
By the \emph{diameter} of $W$, we mean $|\{i \mid 0 \leq i \leq D, E^*_iW\ne 0\}|-1$.
Let $r$ denote the endpoint of $W$ and $d$ the diameter of $W$.
By \cite[Lemma 3.9]{Terwilliger.1992}, we have (\rm{i}) $E^*_iW \ne 0$  if and only if $r \leq i \leq r+d$; (\rm{ii}) $W=\bigoplus^d_{h=0}E^*_{r+h}W$, an orthogonal direct sum.
An irreducible $T$-module $W$ is said to be \emph{thin} whenever $\dim(E^*_iW)\leq 1$ for $0 \leq i \leq D$.
There exists a unique thin irreducible $T$-module with endpoint $0$ and diameter $D$, which we call it the \emph{primary} $T$-module.
The primary $T$-module has a basis $E^*_0\mathbf{j}, \ldots, E^*_D\mathbf{j}$ \cite[Lemma 3.6]{Terwilliger.1992}, where $\mathbf{j}$ is the all-ones vector.

The graph $G$ is said to be \emph{thin with respect to $x$} whenever every irreducible $T(x)$-module is thin.
The graph $G$ is said to be \emph{thin} whenever $G$ is thin with respect to every vertex $x$ of $G$.
See \cite[Section 6]{Terwilliger.1993b} for examples of thin $Q$-polynomial distance-regular graphs.

\bibliographystyle{plain}
\bibliography{survey}
\end{document}